%% file: rproj.tex
\documentclass[oneside,english]{amsart}
\usepackage[T1]{fontenc}
\usepackage[latin9]{inputenc}
\usepackage{float}
\usepackage{amsthm}
\usepackage{amssymb}
\usepackage{graphicx}

\makeatletter

\providecommand{\tabularnewline}{\\}
\floatstyle{ruled}
\newfloat{algorithm}{tbp}{loa}
\providecommand{\algorithmname}{Algorithm}
\floatname{algorithm}{\protect\algorithmname}

\numberwithin{equation}{section}
\numberwithin{figure}{section}
\theoremstyle{plain}
\newtheorem{thm}{\protect\theoremname}[section]
  \theoremstyle{remark}
  \newtheorem{rem}[thm]{\protect\remarkname}
  \theoremstyle{plain}
  \newtheorem{lem}[thm]{\protect\lemmaname}
  \theoremstyle{plain}
  \newtheorem{cor}[thm]{\protect\corollaryname}

\makeatother

\usepackage{babel}
  \providecommand{\corollaryname}{Corollary}
  \providecommand{\lemmaname}{Lemma}
  \providecommand{\remarkname}{Remark}
\providecommand{\theoremname}{Theorem}

\begin{document}
\global\long\def\acs{A_{c}}

\global\long\def\bcs{B_{c}}

\global\long\def\as{A_{s}}

\global\long\def\bs{B_{s}}

\global\long\def\ucs{U_{c}}

\global\long\def\ucshat{\hat{U}_{c}}

\title[On rank reduction of separated representations]{Randomized Interpolative Decomposition of Separated Representations}

\author{David J. Biagioni$\,^{*}$, Daniel Beylkin$\,^{**}$ and Gregory
Beylkin$\,^{***}$}

\curraddr{$^{*}$Department of Aerospace Engineering Sciences \\
 University of Colorado at Boulder \\
 429 UCB \\
 Boulder, CO 80309-0526 \\
$\,^{**}$Program in Applied Mathematics\\
Yale University\\
51 Prospect St.\\
New Haven, CT 06511\\
$^{***}$Department of Applied Mathematics \\
 University of Colorado at Boulder \\
 526 UCB \\
 Boulder, CO 80309-0526 \\
~}

\thanks{D. J. Biagioni was partially supported by NREL grant UGA-0-41026-08
and NSF grant DMS-1228359. D. Beylkin received support under FA9550-11-C-0028
by the DoD, AFOSR NDSEG Fellowship. G. Beylkin was partially supported
by NSF grants DMS-1009951, DMS 1228359 and DOE/ORNL grant 4000038129. }
\begin{abstract}
We introduce an algorithm to compute tensor Interpolative Decomposition
(tensor ID) for the reduction of the separation rank of Canonical
Tensor Decompositions (CTDs). Tensor ID selects, for a user-defined
accuracy $\epsilon$, a near optimal subset of terms of a CTD to represent
the remaining terms via a linear combination of the selected terms.
Tensor ID can be used as an alternative to or in combination with
the Alternating Least Squares (ALS) algorithm. We present examples
of its use within a convergent iteration to compute inverse operators
in high dimensions. We also briefly discuss the spectral norm as a
computational alternative to the Frobenius norm in estimating approximation
errors of tensor ID. 

We reduce the problem of finding tensor IDs to that of constructing
Interpolative Decompositions of certain matrices. These matrices are
generated via randomized projection of the terms of the given tensor.
We provide cost estimates and several examples of the new approach
to the reduction of separation rank.
\end{abstract}


\maketitle

\section{Introduction}

The computational cost of many fast algorithms grows exponentially
in the problem dimension, $d$. In order to maintain linear complexity
in $d$, we use separated representations as a framework for numerical
computations in high dimensions \cite{BEY-MOH:2002,BEY-MOH:2005}.
Separated representations generalize the usual notion of separation
of variables by representing a multivariate function as 
\begin{equation}
u\left(x_{1},x_{2},\dots x_{d}\right)=\sum_{l=1}^{r}\sigma_{l}u_{1}^{(l)}\left(x_{1}\right)u_{2}^{(l)}\left(x_{2}\right)\cdots u_{d}^{(l)}\left(x_{d}\right),\label{eq: separted representation}
\end{equation}
where the number of terms, $r$, is called the separation rank of
the function $u$. Any discretization of (\ref{eq: separted representation})
with $u_{i_{j}}^{(l)}=u_{j}^{(l)}\left(x_{i_{j}}\right)$, $i_{j}=1,\dots M_{j}$
and $j=1,\dots d$, leads to a Canonical Tensor Decomposition (CTD),
\begin{equation}
\mathcal{U}_{i_{1}\dots i_{d}}=\sum_{l=1}^{r}\sigma_{l}u_{i_{1}}^{(l)}u_{i_{2}}^{(l)}\cdots u_{i_{d}}^{(l)}.\label{eq:introCTDelements}
\end{equation}
The functions $u_{j}^{(l)}\left(x_{j}\right)$ in (\ref{eq: separted representation})
and the corresponding vectors $u_{i_{j}}^{(l)}$ in (\ref{eq:introCTDelements})
are normalized to have unit Frobenius norm so that the size of the
terms is carried by their positive $s$-values, $\sigma_{l}$. 

Numerical computations using such representations requires an algorithm
to reduce the number of terms for a given accuracy, $\epsilon$. Such
reduction can be achieved via Alternating Least Squares (ALS) algorithm
(see e.g., \cite{HARSHM:1970,CAR-CHA:1970,BRO:1997,BEY-MOH:2002,BEY-MOH:2005,TOM-BRO:2006,KOL-BAD:2009}).
Specifically, given a CTD of rank $r$, ALS allows us to find a representation
of the same form but with fewer terms, 
\begin{equation}
\tilde{\mathcal{U}}_{i_{1}\dots i_{d}}=\sum_{l=1}^{k}\tilde{\sigma}_{l}\tilde{u}_{i_{1}}^{(l)}\tilde{u}_{i_{2}}^{(l)}\cdots\tilde{u}_{i_{d}}^{(l)},\,\,\,\,\,\,\,\,\,\,\,\, k<r,\label{eq: reduced separated representation via ALS}
\end{equation}
so that $\left\Vert \mathcal{U}-\tilde{\mathcal{U}}\right\Vert \le\epsilon\left\Vert \mathcal{U}\right\Vert $,
where $\epsilon$ is a user-selected accuracy. Standard operations
on separated representations of rank $k$, such as multiplication,
may result in a large number, e.g., $\mathcal{O}\left(k^{2}\right)$,
of intermediate terms. If the intermediate separation rank is $r=\mathcal{O}\left(k^{2}\right)$
reducible to $\mathcal{O}\left(k\right)$, then the cost of ALS can
estimated as $\mathcal{O}\left(d\cdot k^{4}\cdot M\right)\cdot(number\,\, of\,\, iterations)$,
where we assumed that $M_{j}=M$, $j=1,\dots d$. Noting that the
number of iterations is not easily controlled and may be large, the
computational cost of ALS can be significant, although it is linear
in the dimension $d$.

In this paper we describe randomized tensor Interpolative Decomposition
(ID) for the reduction of the separation rank of a tensor in the canonical
form (\ref{eq:introCTDelements}) that is faster (for a class of problems)
than ALS by $\mathcal{O}\left(k\right)\cdot(number\,\, of\,\, iterations)$.
We adopt the term tensor ID by analogy with matrix ID (see e.g., \cite{HA-MA-TR:2011})
and note that the reduction of the number of linearly dependent terms
in the separated representations has been already performed using
Gram matrices.%
\footnote{Martin Mohlenkamp (Ohio University) and G.B. developed and used a
deterministic algorithm for tensor ID based on pivoted Cholesky decomposition
of the Gram matrix (see e.g., discussion around Theorem~\ref{thm:gram-id-error}).%
} The reason we revisit this type of approach is that the use of the
Gram matrix limits the accuracy of computations to about one half
of the available digits (e.g., single precision while using double
precision arithmetic) and may be problematic due to the dynamic range
of its entries if used for sufficiently high dimensions. The randomized
approach allows us to avoid these issues, at least for moderate dimensions.
While we tested the randomized algorithm on a number of representative
examples, the full justification of our approach needs additional
work. Currently we verify accuracy of the result \emph{a posteriori}.
The randomized tensor ID algorithm can be supplemented by a fixed
number of ALS-type iterations which play an auxiliary role of improving
the conditioning of CTD. We demonstrate a significant acceleration
of the reduction of separation rank on a number of examples.

On the technical level, we extend recently developed ideas of randomized
algorithms for matrix ID \cite{C-G-M-R:2005,L-W-M-R-T:2007,MA-RO-TY:2011,HA-MA-TR:2011}
to the tensor setting. These algorithms use the fact that the projection
of columns of a low rank matrix onto a sufficient number of random
vectors provides, with high probability, a good approximation of the
matrix range. The number of random vectors needed is only slightly
greater than the rank of the matrix which, in turn, implies that the
cost of computing the matrix ID can be greatly reduced. We extend
this randomized framework to tensors by posing the problem as that
of computing the matrix ID of a certain associated matrix. We first
present the main ideas without  details or proofs and lay out technical
information in the sections that follow. 

Our interest in the development of tensor ID stems from applications
where we compute functions of operators in high dimensions. In particular,
we are interested in computing Green's functions (see examples in
Section~\ref{sec:Examples}). The inverse operator (the Green's function)
can be computed via self-correcting, quadratically convergent Schulz
iteration \cite{SCHULZ:1933}. Many other important functions of operators
can also be obtained via self-correcting, convergent iterations. In
such cases we can use the iteration itself to generate the necessary
variety of terms in the separated representation and use the randomized
tensor ID as a way to select the desired subset of terms while maintaining
accuracy. The intermediate errors incurred by such reduction are then
corrected by the iteration itself. We dub such an approach Self-Guiding
Tensor Iteration (SGTI) and provide examples of computing Green's
functions using it. We plan to address a general problem of computing
functions of operators in high dimensions via the SGTI approach separately. 

Finally, we have observed that, in a tensor setting, a randomized
approach based on sampling may provide an additional speed-up in reducing
the separation rank. We plan to address this idea separately, as well.

\subsection{Definitions and notation.}

Throughout the paper, CTDs are denoted by the calligraphic letters
$\mathcal{Q}$ through $\mathcal{Z}$. We assume that each direction
$j=1,\dots,d$ may be represented by $M_{j}$ values, stored in the
vectors $\mathbf{u}_{j}$. Thus, $\mathcal{U}\in\mathbb{V}=\bigotimes_{j=1}^{d}\mathbb{R}^{M_{j}}$
and, using the Kronecker product notation, can also be written as
\begin{equation}
\mathcal{U}=\sum_{l=1}^{r}\sigma_{l}\bigotimes_{j=1}^{d}\mathbf{u}_{j}^{(l)}\label{eq:CTD}
\end{equation}
instead of (\ref{eq:introCTDelements}). It may sometimes be convenient
to emphasize that $\mathcal{U}$ is a sum of rank-one tensors, so
that

\begin{equation}
\mathcal{U}=\sum_{l=1}^{r}\sigma_{l}\mathcal{U}^{(l)},\,\,\,\,\,\mbox{where}\,\,\,\,\,\,\,\,\,\,\mathcal{U}^{(l)}=\bigotimes_{j=1}^{d}\mathbf{u}_{j}^{(l)}.\label{eq:intro-sumofrankoneterms}
\end{equation}
The standard Frobenius inner product between any two tensors $\mathcal{U}$
and $\mathcal{V}$ is defined as
\begin{equation}
\langle\mathcal{U},\mathcal{V}\rangle=\sum_{i_{1}=1}^{M_{1}}\cdots\sum_{i_{d}=1}^{M_{d}}\mathcal{U}_{i_{1}\dots i_{d}}\mathcal{V}_{i_{1}\dots i_{d}}\label{eq:intro-frobenius-inner-product-definition}
\end{equation}
which for CTDs reduces to
\begin{equation}
\langle\mathcal{U},\mathcal{V}\rangle=\sum_{l=1}^{r_{u}}\sum_{m=1}^{r_{v}}\sigma_{l}^{u}\sigma_{m}^{v}\langle\mathcal{U}^{(l)},\mathcal{V}^{(m)}\rangle\,\,=\,\,\sum_{l=1}^{r_{u}}\sum_{m=1}^{r_{v}}\sigma_{l}^{u}\sigma_{m}^{v}\prod_{j=1}^{d}\langle\mathbf{u}_{j}^{(l)},\mathbf{v}_{j}^{(m)}\rangle,\label{eq:intro-frobenius-inner-product-for-CTD}
\end{equation}
where $\langle\cdot,\cdot\rangle$ denotes the inner product between
component vectors. The Frobenius norm is then defined as
\begin{equation}
\left\Vert \mathcal{U}\right\Vert _{F}=\sqrt{\langle\mathcal{U},\mathcal{U}\rangle}.\label{eq:intro-frob-norm}
\end{equation}

\begin{rem}
The directional vectors $\mathbf{u}_{j}^{(l)}$ may represent objects
of different types, including proper one dimensional vectors, matrices
or even low dimensional tensors. The vectors $\mathbf{u}_{j}^{(l)}$
can have a complicated structure (e.g., sparse matrices or low dimensional
tensors) as long as the Frobenius inner product between them is well
defined.
\end{rem}

\subsection{Tensor interpolative decomposition\label{sub:intro-Tensor-interpolative-decomposition}}

The tensor ID is motivated by the fact that certain tensor operations,
e.g. multiplication, can result in CTDs with many nearly linearly
dependent terms. In such case, we formulate the problem of separated
rank reduction as that of identifying a near optimal (in a sense to
be explained later) subset $\left\{ \mathcal{U}^{(l_{m})}\right\} _{m=1}^{k}$,
$k<r$, of linearly independent terms of $\mathcal{U}$ in (\ref{eq:CTD})
to represent the remaining terms. In contrast to (\ref{eq: reduced separated representation via ALS}),
we seek a CTD 
\begin{equation}
\mathcal{U}_{k}=\sum_{m=1}^{k}\widehat{\sigma}_{m}\mathcal{U}^{(l_{m})},\,\,\,\,\,\,\,\,\,\, k<r,\label{eq: subset separted representation}
\end{equation}
with modified $s$-values $\widehat{\sigma}_{m}$, so that $\left\Vert \mathcal{U}-\mathcal{U}_{k}\right\Vert \le\epsilon\left\Vert \mathcal{U}\right\Vert $. 

The tensor ID extends the concept of the matrix ID introduced and
developed in \cite{GU-EIS:1996,TYRTYS:1996,GO-ZA-TY:1997,GO-TY-ZA:1997,C-G-M-R:2005,L-W-M-R-T:2007,MA-RO-TY:2011,HA-MA-TR:2011}.
For an $m\times n$ matrix $A$ and desired accuracy $\epsilon$,
the construction of the matrix ID entails identifying a set of columns
with indices $\mathcal{L}_{k}=\{l_{1},l_{2},\dots,l_{k}\}\subseteq\{1,2,\dots,n\}$
and $k\times n$ (well conditioned) coefficient matrix $P$, such
that 
\begin{equation}
\left\Vert A-\acs P\right\Vert _{2}<\epsilon\left\Vert A\right\Vert _{2},\label{eq:intro A-AcP norm}
\end{equation}
where $\acs$ is the so-called column skeleton of $A$, i.e., a matrix
containing the columns with indices $\mathcal{L}_{k}$.

We associate the problem of selecting linearly independent terms of
the tensor $\mathcal{U}$ to that of selecting linearly independent
columns of an appropriately chosen matrix. For conceptual purposes,
let us introduce the $N\times r$ matrix 
\begin{equation}
U=\left[\begin{array}{cccc}
| & | &  & |\\
\sigma_{1}\mathcal{U}^{(1)} & \sigma_{2}\mathcal{U}^{(2)} & \cdots & \sigma_{r}\mathcal{U}^{(r)}\\
| & | &  & |
\end{array}\right],\label{eq:umat}
\end{equation}
by treating the terms of $\mathcal{U}$ as vectors in $\mathbb{R}^{N}$,
where $N=\prod_{j=1}^{d}M_{j}$ is gigantic. Forming such a matrix
is in \emph{no way practical} for most (if not all) problems of interest.
However, by efficiently constructing the matrix ID of $U$ (without
using this matrix explicitly), we show in Section~\ref{sec:A-new-approach-to-the-reduction-of-tensor-rank}
that it is possible to construct a rank-$k$ approximation to the
tensor $\mathcal{U}$ as in (\ref{eq: subset separted representation}). 

In Section~\ref{sub:Tensor-ID-problem}, we start by showing how
to use the $r\times r$ Gram matrix (corresponding to $U^{*}U$) 
\begin{equation}
G=\left[\begin{array}{cccc}
\langle\sigma_{1}\mathcal{U}^{(1)},\sigma_{1}\mathcal{U}^{(1)}\rangle & \langle\sigma_{1}\mathcal{U}^{(1)},\sigma_{2}\mathcal{U}^{(2)}\rangle & \cdots & \langle\sigma_{1}\mathcal{U}^{(1)},\sigma_{r}\mathcal{U}^{(r)}\rangle\\
\langle\sigma_{2}\mathcal{U}^{(2)},\sigma_{1}\mathcal{U}^{(1)}\rangle & \langle\sigma_{2}\mathcal{U}^{(2)},\sigma_{2}\mathcal{U}^{(2)}\rangle & \cdots & \langle\sigma_{2}\mathcal{U}^{(2)},\sigma_{r}\mathcal{U}^{(r)}\rangle\\
\vdots & \vdots & \ddots & \vdots\\
\langle\sigma_{r}\mathcal{U}^{(r)},\sigma_{1}\mathcal{U}^{(1)}\rangle & \langle\sigma_{r}\mathcal{U}^{(r)},\sigma_{2}\mathcal{U}^{(2)}\rangle & \cdots & \langle\sigma_{r}\mathcal{U}^{(r)},\sigma_{r}\mathcal{U}^{(r)}\rangle
\end{array}\right]\label{eq:intro-gmat}
\end{equation}
to compute the tensor ID of $\mathcal{U}$. For this purpose, we use
a symmetric rank-$k$ ID of $G$,
\begin{equation}
G_{k}=P^{*}G_{S}P,\label{eq:P*GP}
\end{equation}
where $G_{S}$ is a $k\times k$ sub-matrix of $G$, and derive estimates
for the resulting error in approximating the tensor $\mathcal{U}$. 

Constructing the tensor ID via the Gram matrix limits the achievable
accuracy $\epsilon$ to about half the digits of machine precision.
To avoid the loss of accuracy, we develop a randomized algorithm for
computing the tensor ID. Specifically (see Section~\ref{sec:A-new-approach-to-the-reduction-of-tensor-rank}),
we form a collection of $\ell$ random tensors in separated form,
\begin{equation}
\mathcal{R}^{(l)}=\bigotimes_{j=1}^{d}\mathbf{r}_{j}^{(l)},\,\,\,\,\,\,\,\,\,\, l=1,\dots,\ell,\label{eq:intro random tensor}
\end{equation}
with independent random variables $r_{i_{j}}^{(l)}$ of zero mean
and unit variance and $\ell$ somewhat larger than the expected separation
rank $k$. We then form the $\ell\times r$ projection matrix
\begin{equation}
Y=\left[\begin{array}{cccc}
\langle\mathcal{R}^{(1)},\sigma_{1}\mathcal{U}^{(1)}\rangle & \langle\mathcal{R}^{(1)},\sigma_{2}\mathcal{U}^{(2)}\rangle & \cdots & \langle\mathcal{R}^{(1)},\sigma_{r}\mathcal{U}^{(r)}\rangle\\
\langle\mathcal{R}^{(2)},\sigma_{1}\mathcal{U}^{(1)}\rangle & \langle\mathcal{R}^{(2)},\sigma_{2}\mathcal{U}^{(2)}\rangle & \cdots & \langle\mathcal{R}^{(2)},\sigma_{r}\mathcal{U}^{(r)}\rangle\\
\vdots & \vdots & \ddots & \vdots\\
\langle\mathcal{R}^{(\ell)},\sigma_{1}\mathcal{U}^{(1)}\rangle & \langle\mathcal{R}^{(\ell)},\sigma_{2}\mathcal{U}^{(2)}\rangle & \cdots & \langle\mathcal{R}^{(\ell)},\sigma_{r}\mathcal{U}^{(r)}\rangle
\end{array}\right],\label{eq:intro-ymat-rproj}
\end{equation}
and use the matrix ID of $Y$ to construct a rank-$k$ approximation
$\mathcal{U}_{k}$ of the tensor $\mathcal{U}$. While the theoretical
underpinnings for this approach require further work, we have found
that the method works well in many practical examples, some of which
which are presented in Section~\ref{sec:Examples}.

\subsection{Contributions and relationship to prior work}

Separated representations of functions and operators were introduced
in \cite{BEY-MOH:2002,BEY-MOH:2005} (see also \cite{BE-GA-MO:2009}).
Since that time, their use as a framework for numerical operator calculus
has appeared in a number of contexts, for example, in the construction
of Green's functions, quadratures, nonlinear approximations, and solutions
of stochastic differential equations (see \cite{GR-KR-TO:2013} for
a recent survey).

The Canonical Tensor Decomposition (CTD) has a long history and is
also known as PARAFAC (PARAllel FACtor analysis) \cite{HARSHM:1970}
or CANDECOMP (CANonical DECOMPosition) \cite{CAR-CHA:1970}. The CTD
has been used extensively in various areas, including chemometrics,
psychometrics, multivariate regression, and signal processing. For
additional references, we refer to recent surveys \cite{KOL-BAD:2009,ACA-BUL:2009,GR-KR-TO:2013}. 

Current practice for reducing the separation rank of a CTD is to use
the Alternating Least Squares (ALS) algorithm or its variants (see,
e.g., \cite{HARSHM:1970,CAR-CHA:1970,BRO:1997,BEY-MOH:2002,BEY-MOH:2005,TOM-BRO:2006,KOL-BAD:2009,MOHLEN:2011}).
Although many problems may indeed be solved using this algorithm,
this approach limits both the size of problems and the attainable
accuracy. This paper provides an efficient alternative to ALS for
applications where a self-correcting convergent iterative algorithm
is used for solving numerical problems (cf., Section~\ref{sec:Examples}). 

We also introduce randomized methods into the CTD setting. The use
of randomized methods for matrix problems has gained wide popularity
within the last several years (see \cite{HA-MA-TR:2011} for a recent
survey). The theoretical foundation for many of these methods can
be traced to the Johnson-Lindenstrauss Lemma \cite{JOH-LIN:1984},
which guarantees so-called concentration of measure when vectors in
high dimensions are projected into a lower dimensional space. The
idea of using random projections in computational problems was proposed
and developed in different areas over several years; see, e.g., \cite{P-T-R-V:1998,ACHLIO:2003,AIL-CHA:2006,CAN-TAO:2006,CA-RO-TA:2006a}.
The application of such an approach to low rank matrix approximation
has been further developed in \cite{MA-RO-TY:2006,SARLOS:2006,L-W-M-R-T:2007,MA-RO-TY:2011}.
A variety of sampling methods for the same purpose have been proposed
in \cite{DRI-KAN:2003,FR-KA-VE:2004,ACH-MCS:2007,RUD-VER:2007}, and
several recent papers introduced randomized methods to the tensor
setting for the Tucker decomposition \cite{V-K-K-V:2005,MAHONE:2006,DRI-MAH:2007,TSOURA:2009}. 

The tensor ID described in Sections~\ref{sub:intro-Tensor-interpolative-decomposition}
and \ref{sec:A-new-approach-to-the-reduction-of-tensor-rank} extends
the concept of the matrix ID developed in \cite{TYRTYS:1996,GO-ZA-TY:1997,GO-TY-ZA:1997}
and further extended in \cite{GU-EIS:1996,C-G-M-R:2005,MA-RO-TY:2006,L-W-M-R-T:2007,MA-RO-TY:2011}.
Our method for constructing tensor ID uses a randomized projection
method with some similarities to approaches in \cite{MA-RO-TY:2006,SARLOS:2006,L-W-M-R-T:2007,MA-RO-TY:2011}. 

Many of the standard arguments to justify the use of random projections
of low rank matrices do not translate easily to the tensor setting
as they lead to overly pessimistic estimates. For example, the usual
arguments on concentration of measure for columns of the matrix $Y$
in (\ref{eq:intro-ymat-rproj}) lead to a lower bound for the number
of random projections that is exponential in the dimension $d$. Similar
pessimistic estimates appears in the recent works \cite{NG-DR-TR:2010,RA-SC-ST:2013}.
Nevertheless, as we describe in Sections~\ref{sec:A-new-approach-to-the-reduction-of-tensor-rank}
and \ref{sec:Examples}, the method we propose works significantly
better than these estimates suggest. Further work on this topic is
needed. 

We begin by reviewing the necessary mathematical preliminaries in
Section~\ref{sec:Preliminaries}. We then describe our approach to
the reduction of separation rank in Section~\ref{sec:A-new-approach-to-the-reduction-of-tensor-rank}.
We set the conceptual framework for the tensor ID in Sections~\ref{sub:id-of-symmetric-matrices}~--~\ref{sub:Tensor-ID-problem}
and describe an efficient randomized algorithm for its computation
in Section~\ref{sec:A-new-approach-to-the-reduction-of-tensor-rank}.
We then address the loss of accuracy and the high cost of evaluating
the Frobenius norm in Section~\ref{sec:Frobenius-norm-and-s-norm},
and discuss rank-one tensor approximation as an alternative method
for norm estimation. Finally, we present several examples illustrating
the new algorithm in Section~\ref{sec:Examples}.

\section{Preliminaries\label{sec:Preliminaries}}

We start by summarizing several randomized algorithms for matrices
that will be used in Section~\ref{sec:A-new-approach-to-the-reduction-of-tensor-rank}.

\subsection{Randomized algorithm for approximating the range of a matrix\label{sub:proto-algorithm}}

Let us define Q-factorization as a decomposition of an $m\times n$
matrix $A$ of fixed rank $k$ via 
\begin{equation}
A_{k}=QS\label{eq:qs-factorization}
\end{equation}
where $Q$ is an $m\times k$ matrix with orthonormal columns and
$S$ is a $k\times n$ coefficient matrix (not necessarily upper-triangular).
An efficient approach is to first construct the matrix $Q$ to approximate
the range of $A$, and then set
\begin{equation}
S=Q^{*}A.\label{eq: s matrix in qs factorization}
\end{equation}
The matrix $A_{k}$ is then seen to be an orthogonal projection of
the columns of $A$ onto the subspace captured by $Q$. 

When $A$ is low rank, the construction of (\ref{eq:qs-factorization})
is amenable to randomized methods. We use Algorithm~4.1 of \cite{HA-MA-TR:2011}
to construct $Q$, summarized in this paper as Algorithm~\ref{alg:Randomized-algorithm-for-approximating-the-range}
(see also \cite{SARLOS:2006,L-W-M-R-T:2007,MA-RO-TY:2011}). 

\begin{algorithm}[H]
\begin{raggedright}
\caption{Randomized Algorithm for Approximating the Range of a Matrix \cite{HA-MA-TR:2011}\label{alg:Randomized-algorithm-for-approximating-the-range}}

\par\end{raggedright}

\begin{raggedright}
Input: An $m\times n$ matrix $A$ of fixed rank $k$, and integer
$\ell>k$.
\par\end{raggedright}

\begin{raggedright}
Output: An $m\times\ell$ matrix $Q$ whose columns comprise an orthonormal
basis for the range of $A$.
\par\end{raggedright}
\begin{enumerate}
\item \begin{raggedright}
Generate an $m\times\ell$ Gaussian random matrix, $R$.
\par\end{raggedright}
\item \begin{raggedright}
Form the $m\times\ell$ matrix, $Y=AR$.
\par\end{raggedright}
\item \raggedright{}Construct QR factorization of $Y$, yielding the matrix
$Q$.\end{enumerate}
\end{algorithm}

\noindent This procedure yields a matrix $Q$ satisfying the following
error estimate,
\begin{thm}
\cite{HA-MA-TR:2011} (Halko, Martinsson, and Tropp) Select a target
rank $k\geq2$ and an integer $p\geq2$, where $\ell=k+p\leq\min\{m,n\}$
and generate an $m\times\ell$ matrix $Q$ with orthonormal columns.
Then we have
\begin{equation}
\mathbb{E}\left[\left\Vert A-QQ^{*}A\right\Vert _{2}\right]\leq\tau_{k+1}(A)\left(1+\frac{4\sqrt{\ell}}{p-1}\sqrt{\min\{m,n\}}\right),\label{eq:proto-alg-expected-error}
\end{equation}
where $\mathbb{E}$ denotes expectation with respect to the random
matrix $R$ and $\tau_{k+1}(A)$ is the $k+1$ singular value of $A$.
\end{thm}
The probability of violating the estimate (\ref{eq:proto-alg-expected-error})
decays exponentially in $p$, 
\begin{equation}
\mbox{Pr}\left\{ \left\Vert A-QQ^{*}A\right\Vert \geq\tau_{k+1}(A)\left(1+9\sqrt{\ell}\cdot\sqrt{\min\{m,n\}}\right)\right\} \leq3\cdot p^{-p},\label{eq: randomize qs probability}
\end{equation}
implying that $p$ need not be large (e.g., $p=5$ or $p=10$). Thus,
with very high probability, the randomized method comes close to constructing
an optimal rank-$k$ approximation (the best possible bound in (\ref{eq:proto-alg-expected-error})
is $\tau_{k+1}(A)$ which follows from the optimality of the SVD).

Let $c_{A}$ denote the cost of applying $A$ to a vector, and let
$t_{R}$ denote the cost of generating a single random entry of $R$.
Constructing $S$ costs $\ell\cdot c_{A}$ operations, leading to
a total complexity on the order of

\begin{equation}
\begin{array}{ccccccc}
\ell\cdot n\cdot t_{R} & + & \ell\cdot c_{A} & + & \ell^{2}\cdot m & + & \ell\cdot c_{A}.\\
\mbox{Form}\,\, R &  & \mbox{Form}\,\, Y &  & \mbox{QR of}\,\, Y &  & \mbox{\mbox{Form}\,}S
\end{array}\label{eq:proto-alg-cost}
\end{equation}

\subsection{Interpolative matrix decomposition\label{sub:Interpolative-matrix-decompositi}}

Our approach also relies on computing the interpolative decomposition
of a matrix. The idea of matrix ID is to find, for a given accuracy
$\epsilon$, a near optimal set of columns (rows) of a matrix so that
the rest of columns (rows) can be represented as a linear combination
from the selected set. Algorithmically, this decomposition proceeds
via pivoted QR factorization and so is closely related to the factorization
(\ref{eq:qs-factorization}) in Section~\ref{sub:proto-algorithm}.
The fact that a basis for the range of $A$ is constructed purely
as a subset of its columns (not necessarily orthonormal) will be crucial
when dealing with CTDs.
\begin{lem}
\label{lem:matrix-id-lemma}\cite{MA-RO-TY:2011}(Martinsson, Rokhlin,
and Tygert) Suppose $A$ is an $m\times n$ matrix. Then, for any
positive integer $k$ with $k\leq m$ and $k\leq n$, there exist
a real $k\times n$ matrix $P$, and a real $m\times k$ matrix $\acs$
whose columns constitute a subset of the columns of $A$, such that

1. Some subset of the columns of $P$ makes up the $k\times k$ identity
matrix,

2. no entry of $P$ has an absolute value greater than 1,

3. $\left\Vert P\right\Vert _{2}\leq\sqrt{k(n-k)+1},$

4. the smallest singular value of $P$ is at least 1,

5. $\acs P=A$ when $k=m$ or $k=n$, and

6. $\left\Vert \acs P-A\right\Vert _{2}\leq\sqrt{k(n-k)+1}\tau_{k+1}(A)$
when $k<m$ and $k<n$, where $\tau_{k+1}(A)$ is the $k+1$ singular
value of $A$.
\end{lem}
Properties 1, 2, 3, and 4 ensure that the interpolative decomposition
of $A$ is numerically stable. 
\begin{rem}
In \cite{C-G-M-R:2005}, the authors show how the rank-revealing QR
factorization of $A$ is used to compute decomposition
\begin{equation}
A=\acs\left[\begin{array}{c|c}
I & T\end{array}\right]P_{c}^{*}+X,\label{eq:column-matrix-id}
\end{equation}
where $\acs$ is the column skeleton of $A$, $P=\left[\begin{array}{c|c}
I & T\end{array}\right]P_{c}^{*}$ and $X=\acs P-A$ as in Lemma \ref{lem:matrix-id-lemma}. The matrix
$P_{c}$ is an $n\times n$ permutation matrix and $T$ is a $k\times(n-k)$
matrix that solves an associated linear system (see \cite{C-G-M-R:2005}
for details). This decomposition can be extended to include both rows
and columns, yielding a $k\times k$ skeleton matrix $\as$ such that
\begin{equation}
A=P_{r}^{*}\left[\begin{array}{c}
I\\
\hline S
\end{array}\right]\as\left[\begin{array}{c|c}
I & T\end{array}\right]P_{c}+X.\label{eq:full-skeleton-matrix-id}
\end{equation}
Here $P_{r}$ is an $m\times m$ permutation matrix and $S$ is a
$(m-k)\times k$ matrix analogous to $T$. While the full decomposition
(\ref{eq:full-skeleton-matrix-id}) is used in the proof of Theorem~\ref{thm:gram-id-error}
(see the Online Supplement), the column oriented decomposition (\ref{eq:column-matrix-id})
is the main ingredient in the tensor ID described in this paper.
\end{rem}

\subsection{A randomized algorithm for interpolative matrix decomposition\label{sub:A-randomized-algorithm-for-matrix-id}}

The deterministic calculation of the interpolative decomposition is
as expensive as the standard QR. However, when $A$ is of low rank,
the matrix ID can be constructed efficiently using a randomized algorithm
described in \cite{L-W-M-R-T:2007,W-L-R-T:2008,MA-RO-TY:2011,HA-MA-TR:2011}
under slightly weaker conditions than those in Lemma~\ref{lem:matrix-id-lemma}.
The randomized matrix ID algorithm is summarized in this paper as
Algorithm~\ref{alg:Randomized-matrix-ID}. 

\begin{algorithm}[H]
\begin{raggedright}
\caption{Randomized Algorithm for Matrix ID \cite{MA-RO-TY:2011}\label{alg:Randomized-matrix-ID}}

\par\end{raggedright}

\begin{raggedright}
Input: An $m\times n$ matrix $A$ and integer $\ell>k$.
\par\end{raggedright}

\begin{raggedright}
Output: Indices $\mathcal{L}_{k}$ of the $k$ skeleton columns, the
$k\times n$ coefficient matrix $P$, and the $m\times k$ column
skeleton matrix $\acs$.
\par\end{raggedright}
\begin{enumerate}
\item \begin{raggedright}
Generate an $m\times\ell$ Gaussian random matrix, $R$.
\par\end{raggedright}
\item \begin{raggedright}
Form the $\ell\times n$ matrix, $Y=R^{*}A$.
\par\end{raggedright}
\item \begin{raggedright}
Construct QR factorization of $Y$.
\par\end{raggedright}
\item \begin{raggedright}
Using Lemma~\ref{lem:matrix-id-lemma}, construct $\mathcal{L}_{k}$
and $P$ from the QR of $Y$.
\par\end{raggedright}
\item \raggedright{}Collect the $k$ skeleton columns into the matrix $\acs$.\end{enumerate}
\end{algorithm}

Steps~1--3 of Algorithm~\ref{alg:Randomized-matrix-ID} are the
same as those of Algorithm~\ref{alg:Randomized-algorithm-for-approximating-the-range}
applied to $A^{*}$; the remainder of the steps reorganize the QR
factorization of the projection matrix $Y$. The cost of collecting
the $k$ columns of $A$ into matrix $\acs$ requires $\mathcal{O}(k\cdot m)$
operations. Hence, the computational cost of the randomized ID is
estimated as 
\begin{equation}
\begin{array}{ccccccc}
\ell\cdot m\cdot t_{R} & + & \ell\cdot c_{A} & + & k\cdot\ell\cdot n & + & k\cdot m.\\
\mbox{Form}\,\, R &  & \mbox{Form}\,\, Y &  & \mbox{ID of}\,\, Y &  & \mbox{Form}\,\,\acs
\end{array}\label{eq:matrix-id-cost}
\end{equation}

To provide an informal description of the properties of the randomized
matrix ID, we state
\begin{lem}
\label{lem:(Observation-3.3)}(Observation 3.3 of \cite{MA-RO-TY:2011})
The randomized matrix ID algorithm constructs matrices $\acs$ and
$P$ such that

1. some subset of the columns of $P$ makes up the $k\times k$ identity
matrix

2. no entry of $P$ has absolute values greater than $2$,

3. $\left\Vert P\right\Vert _{2}\leq\sqrt{4k\left(n-k\right)+1}$,

4. the smallest singular value of $P$ is at least 1,

5. $\acs P=A$ when $k=m$ or $k=n$,

6. $\left\Vert \acs P-A\right\Vert _{2}\leq\sqrt{4k\left(n-k\right)+1}\tau_{k+1}(A)$
when $k<m$ and $k<n$, where $\tau_{k+1}(A)$ is the $k+1$ singular
value of $A$.
\end{lem}
Comparing this lemma with Lemma~\ref{lem:matrix-id-lemma}, the only
difference is the appearance of a extra factor of size $\sim2$ in
Properties 2, 3 and 6.
\begin{rem}
The proofs of the results summarized above assume that the random
variables are Gaussian with zero mean and unit variance. However,
it is well known that effectiveness of many randomized algorithms
is somewhat independent of the distribution. Numerical experiments
show that uniform, Bernoulli, log normal and several other distributions
also work well in this context. This observation is an important for
the discussion of randomized tensor projections (see Section~\ref{sub:Random-tensors-in-canonical form}).
\end{rem}

\subsection{Q-factorization for tensors\label{sub:intro-Q-factorization}}

Collecting together the vectors $\mathbf{u}_{j}^{(l)}$ in (\ref{eq:CTD})
from all terms $l=1,\dots,r$, we define the $M_{j}\times r$ directional
component matrices,
\begin{equation}
U_{(j)}=\left[\begin{array}{cccc}
| & | &  & |\\
\mathbf{u}_{j}^{(1)} & \mathbf{u}_{j}^{(2)} & \cdots & \mathbf{u}_{j}^{(r)}\\
| & | &  & |
\end{array}\right],\,\,\,\,\,\,\,\,\,\,\,\,\, j=1,\dots,d.\label{eq:component-matrices}
\end{equation}
We refer to the component index $j$ as the direction -- rather than
dimension -- since these components are not necessarily univariate. 

Factorizing the directional component matrices of the tensor $\mathcal{U}$
leads to a factorization of $\mathcal{U}$ itself. Consider 
\begin{equation}
U_{(j)}=Q_{(j)}S_{(j)},\,\,\,\,\,\,\,\,\,\, j=1,\dots,d,\label{eq:introQS}
\end{equation}
where $Q_{(j)}$ is an $M_{j}\times k_{j}$ matrix with orthonormal
columns and $S_{(j)}$ is a $k_{j}\times r$ directional component
matrix,
\begin{equation}
S_{(j)}=\left[\begin{array}{cccc}
| & | &  & |\\
\mathbf{s}_{j}^{(1)} & \mathbf{s}_{j}^{(2)} & \cdots & \mathbf{s}_{j}^{(r)}\\
| & | &  & |
\end{array}\right],\,\,\,\,\,\,\,\,\,\,\,\,\, j=1,\dots,d.\label{eq:component-matrices-1}
\end{equation}
Here $k_{j}$ is the rank of the matrix $U_{(j)}$ (for a given accuracy
$\epsilon$). Given (\ref{eq:introQS}), it is straightforward to
demonstrate that tensor $\mathcal{U}$ in (\ref{eq:CTD}) admits the
decomposition 
\begin{equation}
\mathcal{U}=\mathcal{Q}\mathcal{S},\,\,\,\,\,\,\,\,\,\,\,\,\,\mbox{where}\,\,\,\,\,\,\,\,\,\,\,\,\mathcal{Q}=\bigotimes_{j=1}^{d}Q_{(j)},\,\,\,\,\,\,\,\,\,\,\mathcal{S}=\sum_{l=1}^{r}\sigma_{l}\bigotimes_{j=1}^{d}\mathbf{s}_{j}^{(l)}.\label{eq:intro-QS}
\end{equation}
We call this -- the $d$ independent factorizations of the component
matrices $U_{(j)}$ -- the Q-factorization of tensor $\mathcal{U}$.
Notice that $\mathcal{Q}$ is a separation rank-one operator, and
$\mathcal{S}$ is a CTD that typically has significantly smaller component
dimensions than $\mathcal{U}$. Since $\mathcal{S}\in\bigotimes_{j=1}^{d}\mathbb{R}^{k_{j}}$
with $k_{j}\leq r$, for many problems of practical interest, we have
$k_{j}\ll M_{j}$ for some or all of the directions $j=1,\dots d$.
We note that Q-factorization is not new and is known under different
names, e.g., CANDELINC or simply compression (see Section 5 of \cite{KOL-BAD:2009}
and references therein). 

It is easy to show 
\begin{lem}
\label{lem:q-is-semi-unitary}Suppose that a rank-$r$ canonical tensor
admits the Q-factorization ~ $\mathcal{U}=\mathcal{Q}\mathcal{S}$
as in (\ref{eq:intro-QS}). Then we obtain 
\begin{equation}
\left\Vert \mathcal{U}\right\Vert _{F}=\left\Vert \mathcal{S}\right\Vert _{F}.\label{eq:s-and-u-norms-equal}
\end{equation}

\end{lem}
Due to this Lemma, the accuracy of the rank reduction for the tensor
$\mathcal{S}$ is the same as that for the original tensor $\mathcal{U}$.
To see this, let $\mathcal{S}_{k}=\sum_{l=1}^{k}\hat{\sigma}_{l}\bigotimes_{j=1}^{d}\hat{\mathbf{s}}_{j}^{(l)}$
and assume it satisfies
\begin{equation}
\left\Vert \mathcal{S}-\mathcal{S}_{k}\right\Vert _{F}\leq\epsilon\left\Vert \mathcal{S}\right\Vert _{F}.\label{eq: s minus sk norm}
\end{equation}
Then, setting $\mathcal{U}_{k}=\mathcal{Q}\mathcal{S}_{k}$, we have
\begin{equation}
\left\Vert \mathcal{U}-\mathcal{U}_{k}\right\Vert _{F}=\left\Vert \mathcal{Q}(\mathcal{S}-\mathcal{S}_{k})\right\Vert _{F}\,\,=\,\,\left\Vert \mathcal{S}-\mathcal{S}_{k}\right\Vert _{F}\,\,\leq\,\,\epsilon\left\Vert \mathcal{U}\right\Vert _{F}.\label{eq: u minus uk norm}
\end{equation}

\section{Randomized Tensor Interpolative Decomposition \label{sec:A-new-approach-to-the-reduction-of-tensor-rank}}

\subsection{On the interpolative decomposition of symmetric matrices\label{sub:id-of-symmetric-matrices}}

We briefly describe a matrix ID for symmetric matrices because it
parallels the development for tensors (and, to our knowledge, its
description is not available in the literature). Suppose $A$ is an
$m\times n$ matrix of rank $k$ (for a given accuracy $\epsilon$)
where $m\gg n$, and let $B=A^{*}A$ denote its $n\times n$ Gram
matrix (see \ref{eq:umat} and \ref{eq:intro-gmat}). Theorem~\ref{thm:gram-id}
below relates the error in approximating $B$ by a symmetric matrix
ID to that of approximating $A$ by a corresponding matrix ID that
uses only the decomposition of $B$. The error estimate is of interest
in cases where the cost of constructing the matrix ID of $B$ is relatively
small compared with that of $A$.
\begin{thm}
\label{thm:gram-id}Suppose $A$ is an $m\times n$ matrix of rank
$k$. Then the $n\times n$ matrix $B=A^{*}A$ admits a symmetric
interpolative decomposition of the form 
\begin{equation}
B_{k}=P^{*}\bs P,\label{eq: symmetric id of gram matrix in theorem}
\end{equation}
where $\bs=\acs^{*}\acs$. The $m\times k$ column skeleton $\acs$
of $A$, and $k\times n$ coefficient matrix $P$ can be computed
using only $B$. In addition, 
\begin{equation}
\left\Vert A-A_{c}P\right\Vert _{2}=\left\Vert B-B_{k}\right\Vert _{2}^{1/2}\leq C(n,k)\tau_{k+1}(A),\label{eq: gram id error estimate}
\end{equation}
where
\begin{equation}
C(n,k)=\left(4k(n-k)+1\right)^{1/4}\left(1+\sqrt{nk(n-k)}\right)^{1/2}=\mathcal{O}\left(n^{3/4}\right),\label{eq: gram id error estimate order of constant}
\end{equation}
and $\tau_{k+1}(A)$ is the $k+1$ singular value of $A$.
\end{thm}
The proof of Theorem~\ref{thm:gram-id} is similar to that of Theorem
3 in \cite{C-G-M-R:2005} and appears in the Online Supplement. 

Theorem~\ref{thm:gram-id} allows us to compute the interpolative
decomposition of $A$ via its Gram matrix, $B$. However, for the
matrix ID of $A$ to have accuracy $\epsilon$, we have to compute
the ID of $B$ to an accuracy of $\epsilon^{2}$. This presents a
limitation since, if $\epsilon^{2}<\epsilon_{machine}$, this approach
will select all columns of $B$ . Thus, in general, the ID of $A$
can be computed in this manner only for accuracies $\sqrt{\epsilon_{machine}}<\epsilon$. 
\begin{rem}
An exception to the last statement occurs if $B$ happens to be a
structured matrix of a particular type (see \cite{DEMMEL:1999}) since
then its singular values can be computed with relative precision.
For example, if the univariate functions in separated representation
(\ref{eq: separted representation}) are exponentials or Gaussians,
then $B$ is a Cauchy matrix which is the key example in \cite{DEMMEL:1999}.
\end{rem}

\subsection{Tensor ID problem\label{sub:Tensor-ID-problem}}

Our goal is to reduce the tensor ID problem to that of skeletonization
of certain matrices and to relate the error of the tensor ID to that
of the corresponding matrix ID. A CTD in (\ref{eq:introCTDelements})
can always be interpreted as a dense tensor with $N=\prod_{j=1}^{d}M_{j}$
elements and represented as the $N\times r$ matrix $U$ in (\ref{eq:umat}).
While the formation of such a matrix is in\textit{ no way practical},
it allows us to establish a connection between the matrix and tensor
settings. 

Assume for a moment that we are able to compute the matrix ID of $U$
directly as well as the corresponding tensor ID of $\mathcal{U}$.
As already mentioned in Section~\ref{sub:intro-Tensor-interpolative-decomposition},
given a rank-$k$ matrix ID of $U$, we can construct a rank-$k$
approximation $\mathcal{U}_{k}$ via
\begin{equation}
\mathcal{U}_{k}=\sum_{m=1}^{k}\alpha_{m}\mathcal{U}^{(l_{m})},\,\,\,\,\,\,\,\,\,\,\,\,\,\alpha_{m}=\sigma_{m}\sum_{l=1}^{r}P_{ml},\label{eq:ID-coefficients}
\end{equation}
where $l_{m}\in\mathcal{L}_{k}$ are the $k$ indices of the skeleton
columns of $U$ and $P$ is the $k\times r$ coefficient matrix. Each
of the original rank-one terms $\mathcal{U}^{(l)}$ of $\mathcal{U}$
in (\ref{eq:CTD}) has its own representation via the skeleton terms,
\begin{eqnarray}
\mathcal{U}_{k}^{(l)} & = & \sum_{m=1}^{k}P_{ml}\mathcal{U}^{(l_{m})},\,\,\,\,\,\,\,\,\,\,\,\,\,\,\, l=1,\dots,r.\label{eq:u-term-as-vector}
\end{eqnarray}
Now define the $N\times r$ matrix
\begin{equation}
U-U_{k}=\left[\begin{array}{cccc}
| & | &  & |\\
\sigma_{1}\mathcal{U}^{(1)}-\mathcal{U}_{k}^{(1)} & \sigma_{2}\mathcal{U}^{(2)}-\mathcal{U}_{k}^{(2)} & \cdots & \sigma_{r}\mathcal{U}^{(r)}-\mathcal{U}_{k}^{(r)}\\
| & | &  & |
\end{array}\right],\label{eq: matrix u minus uk}
\end{equation}
where tensors are interpreted as vectors in $\mathbb{R}^{N}$ (but
are not necessarily of separation rank one). The error of the tensor
ID can be bounded from above by the error in the matrix ID of $U$.
\begin{thm}
\label{thm:tensor-id-error-estimate}Suppose~ $\mathcal{U}$ is a
rank-$r$ canonical tensor, and let $U$ denote the $N\times r$ matrix
in (\ref{eq:umat}). Further, suppose that a rank-$k$ tensor ID of~
$\mathcal{U}_{k}$ of~ $\mathcal{U}$ is given as in (\ref{eq:ID-coefficients}).
Then we have 
\begin{equation}
\left\Vert \mathcal{U}-\mathcal{U}_{k}\right\Vert _{F}\leq\sqrt{r}\left\Vert U-U_{k}\right\Vert _{F}\leq r\left\Vert U-U_{k}\right\Vert _{2}.\label{eq: ctds u minus uk errors}
\end{equation}

\end{thm}
For the proof, we need the following 
\begin{lem}
\label{lem:tensor matrix frob norm inequality}For any rank-$r$ canonical
tensor $\mathcal{U}$ and corresponding matrix~ $U$  in (\ref{eq:umat}),
we have
\begin{equation}
\left\Vert \mathcal{U}\right\Vert _{F}\leq\sqrt{r}\left\Vert U\right\Vert _{F}.\label{eq: tensor matrix frob norm inequality}
\end{equation}
\end{lem}
\begin{proof}
By definition, we have
\[
\left\Vert U\right\Vert _{F}^{2}=\sum_{i_{1}=1}^{M_{1}}\cdots\sum_{i_{d}=1}^{M_{d}}\sum_{l=1}^{r}\left(\sigma_{l}\mathcal{U}_{i_{1}\dots i_{d}}^{(l)}\right)^{2}.
\]
Since any real numbers $\{a_{l}\}_{l=1}^{r}$ satisfy $\left(\sum_{l=1}^{r}a_{l}\right)^{2}\leq r\sum_{l=1}^{r}a_{l}^{2}$,
it follows that 
\begin{eqnarray*}
\left\Vert \mathcal{U}\right\Vert _{F}^{2} & = & \sum_{i_{1}=1}^{M_{1}}\cdots\sum_{i_{d}=1}^{M_{d}}\mathcal{U}_{i_{1}\dots i_{d}}^{2}\,\,\,=\,\,\,\sum_{i_{1}=1}^{M_{1}}\cdots\sum_{i_{d}=1}^{M_{d}}\left(\sum_{l=1}^{r}\sigma_{l}\mathcal{U}_{i_{1}\dots i_{d}}^{(l)}\right)^{2}\\
 & \le & r\sum_{i_{1}=1}^{M_{1}}\cdots\sum_{i_{d}=1}^{M_{d}}\sum_{l=1}^{r}\left(\sigma_{l}\mathcal{U}_{i_{1}\dots i_{d}}^{(l)}\right)^{2}\,\,\,=\,\,\, r\left\Vert U\right\Vert _{F}^{2}.
\end{eqnarray*}

\end{proof}
To prove Theorem~\ref{thm:tensor-id-error-estimate}, we observe
that the first inequality follows directly from Lemma~\ref{lem:tensor matrix frob norm inequality}.
Letting $\tau_{1}\ge\tau_{2}\ge\dots\ge\tau_{r}\ge0$ denote the eigenvalues
of the $r\times r$ matrix $(U-U_{k})^{*}(U-U_{k})$, we have
\[
\left\Vert U-U_{k}\right\Vert _{F}^{2}=\mbox{tr}\left(U-U_{k}\right){}^{*}\left(U-U_{k}\right)\,\,=\,\,\sum_{j=1}^{r}\tau_{j}\,\,\leq\,\, r\tau_{1}\,\,=\,\, r\left\Vert U-U_{k}\right\Vert _{2}^{2}.
\]

For computing the tensor ID, instead of the $N\times r$ matrix $U$
we can use the $r\times r$ Gram matrix. The next theorem relates
the error in the tensor ID of $\mathcal{U}$ computed via its Gram
matrix to that of the matrix ID of the Gram matrix itself. 
\begin{thm}
\label{thm:gram-id-error}Suppose $G$ is the $r\times r$ Gram matrix
defined in (\ref{eq:intro-gmat}). Let a symmetric interpolative decomposition
of $G$ be given by 
\begin{equation}
G_{k}=P^{*}G_{s}P,\label{eq: symmetric id of gram matrix}
\end{equation}
where $\left\Vert G-G_{k}\right\Vert _{2}\leq\epsilon_{k}$, $G_{s}$
is a $k\times k$ sub-matrix of $G$, and $P$ is a $k\times r$ coefficient
matrix. Then the corresponding tensor ID~ $\mathcal{U}_{k}$ of $\mathcal{U}$
satisfies 
\begin{equation}
\left\Vert \mathcal{U}-\mathcal{U}_{k}\right\Vert _{F}\leq r\sqrt{\epsilon_{k}}.\label{eq: gram id error bound}
\end{equation}
\end{thm}
\begin{proof}
The proof is a straightforward consequence of Theorem~\ref{thm:gram-id}
and is obtained by setting $A=U$ and $B=G$, and using Theorem~\ref{thm:tensor-id-error-estimate}. \end{proof}
\begin{rem}
The matrix $P$ appearing in Theorem~\ref{thm:gram-id-error} coincides
with that of $P_{r}$ in (\ref{eq:full-skeleton-matrix-id}). In other
words, computing the symmetric ID of $G$ requires computing the full
skeletonization. We refer the reader to the proof in the Online Supplement
for additional details.
\end{rem}

\subsection{Projection algorithm with random tensors in canonical form\label{sub:A-method-based-on-canonical-random-tensors}}

\subsubsection{Random tensors in canonical form\label{sub:Random-tensors-in-canonical form}}

In order to form the matrix $Y$ in (\ref{eq:intro-ymat-rproj}),
we generate random rank-one tensors in canonical form, 
\begin{equation}
\mathcal{R}=\bigotimes_{j=1}^{d}\mathbf{r}_{j}=\prod_{j=1}^{d}r_{i_{j}},\,\,\,\,\,\,\, i_{j}=1,\dots,M_{j}.\label{eq:random-tensor}
\end{equation}
If the entries of $\mathbf{r}_{j}$ are realizations of independent
random variables with zero mean and unit variance, then each entry
of $\mathcal{R}$ has zero mean and unit variance as well: 
\begin{equation}
\mathbb{E}\left[\mathcal{R}_{i_{1}\dots i_{d}}\right]=\mathbb{E}\left[\prod_{j=1}^{d}r_{i_{j}}\right]=\prod_{j=1}^{d}\mathbb{E}\left[r_{i_{j}}\right]=0,\label{eq: expected value of random tensor element}
\end{equation}
and 
\begin{equation}
\mathbb{E}\left[\left(\mathcal{R}_{i_{1}\dots i_{d}}-\mathbb{E}\left[\mathcal{R}_{i_{1}\dots i_{d}}\right]\right)^{2}\right]=\mathbb{E}\left[\mathcal{R}_{i_{1}\dots i_{d}}^{2}\right]=\prod_{j=1}^{d}\mathbb{E}\left[\left(r_{i_{j}}\right)^{2}\right]=1.\label{eq: variance of random tensor element}
\end{equation}
In our numerical experiments, we consider several distributions for
the entries of the vectors $\mathbf{r}_{j}$, specifically,
\begin{itemize}
\item Normal
\item Uniform
\item Bernoulli 
\item $1/d$ power distribution, i.e., $r_{i_{j}}=\mbox{sign}\left(\tilde{r}_{i_{j}}\right)\times\left|\tilde{r}_{i_{j}}\right|^{1/d}$,
where, e.g., $\tilde{r}_{i_{j}}\sim N(0,1)$
\end{itemize}
Algorithm~\ref{alg:Tensor-ID-via-random-projections} below shows
the steps for constructing the tensor ID via random projections. As
before, $\ell$ is an integer slightly larger than the expected separation
rank of the approximation, $k$.

\begin{algorithm}[H]
\caption{Tensor ID via Random Projection \label{alg:Tensor-ID-via-random-projections}}

\begin{raggedright}
Input: A rank-$r$ CTD $\mathcal{U}$ and anticipated separation rank
$\ell>k$.
\par\end{raggedright}

\begin{raggedright}
Output: A rank-$k$ tensor ID, $\mathcal{U}_{k}$.
\par\end{raggedright}
\begin{enumerate}
\item Form the $\ell$ random CTDs $\mathcal{R}^{(l)}$ as in (\ref{eq:random-tensor}).
\item Form the $\ell\times r$ projection matrix $Y$ as in (\ref{eq:intro-ymat-rproj}).
\item Compute the rank-$k$ matrix ID of $Y$ via Algorithm~\ref{alg:Randomized-matrix-ID}.
\item Form the rank-$k$ tensor ID $\mathcal{U}_{k}$ via (\ref{eq:ID-coefficients}).
\end{enumerate}
\end{algorithm}

\subsubsection{Cost estimate}

We estimate the cost of the random projection method using $\ell$
random tensors assuming that $M_{1}=\cdots=M_{d}=M$. As we did in
Section~\ref{sub:proto-algorithm}, let $t_{R}$ denote the cost
of generating a single random number used to construct $\mathcal{R}^{(l)}$,
$l=1,\dots\ell$. Forming the collection of $\ell$ random tensors
costs $\mathcal{O}(d\cdot\ell\cdot t_{R}\cdot M)$. Computing each
element of matrix $[Y]_{lm}=\langle\mathcal{R}^{(l)},\sigma_{m}\mathcal{U}^{(m)}\rangle$
requires $\mathcal{O}(d\cdot M)$ operations so that formation of
$Y$ takes $\mathcal{O}(d\cdot\ell\cdot r\cdot M)$ operations. The
cost of obtaining the matrix ID of $Y$ is given in Section~\ref{sub:A-randomized-algorithm-for-matrix-id}.
Finally, constructing the rank-$k$ tensor ID by collecting the $k$
rank-one terms of $\mathcal{U}$ requires $\mathcal{O}(d\cdot k\cdot M)$
operations, and computing the new $s$-values, an additional $\mathcal{O}(k\cdot r)$
operations. Hence, the total cost is on the order of
\begin{equation}
\begin{array}{ccccccc}
d\cdot\ell\cdot M\cdot t_{R} & + & d\cdot\ell\cdot r\cdot M & + & k\cdot\ell\cdot r & + & k\cdot(d\cdot M+r).\\
\mbox{Form tensors\,}\mathcal{R}^{(l)} &  & \mbox{Form}\, Y &  & \mbox{ID of}\, Y &  & \mbox{Form}\,\,\mathcal{U}_{k}
\end{array}\label{eq:computational-cost-rproj}
\end{equation}

\subsubsection{Comparison of computational costs\label{sub:Comparison-of-computational-costs}}

For comparison of the computational complexity of the algorithms,
let us consider the case where the nominal (original) separation rank
$r\sim\mathcal{O}(k^{2})$ is reducible to $\mathcal{O}(k)$ via tensor
ID and $M\gg r$ is the same in all directions. This situation commonly
occurs when, e.g., two functions are multiplied together, or when
an operator is applied to a function represented in separated form.
The situation where $M\gg r$ is typical when the component vectors
$\mathbf{u}_{j}^{(l)}$ are two- or three-dimensional. For simplicity,
let us ignore the cost of generating a random numbers since this cost
is not significant.

Before comparing the algorithms, we remark on the roles that Q-factorization
and tensor ID play in reducing the size of relevant computational
parameters. After performing Q-factorization on a given tensor, the
size $M$ is reduced to at most $r$. Therefore, in all of the estimates
below we replace $M$ by $k^{2}$. 

Combining these assumptions with the estimates (\ref{eq:computational-cost-rproj}),
we arrive at computational complexities shown in Table~\ref{tab:Computational-complexity-estimat}.
The tensor ID approach is faster than ALS by a remarkable factor $k\cdot(number\,\, of\,\, iterations)$,
where the number of iterations in ALS method is quite large.

\begin{table}[H]
\begin{centering}
\begin{tabular}{|l|l|}
\hline 
\textbf{Reduction method} & \textbf{Computational complexity}\tabularnewline
\hline 
ALS & $\left(d\cdot k^{4}\cdot M\right)\cdot(number\,\, of\,\, iterations)\sim$\tabularnewline
 & $\left(d\cdot k^{6}\right)\cdot(number\,\, of\,\, iterations)$\tabularnewline
\hline 
Tensor ID: random projection & $d\cdot k^{3}\cdot M+k^{4}\sim d\cdot k^{5}$\tabularnewline
\hline 
\multicolumn{1}{l}{} & \multicolumn{1}{l}{}\tabularnewline
\end{tabular}
\par\end{centering}

\caption{Estimates of computational complexity. \label{tab:Computational-complexity-estimat}}

\end{table}

\section{Frobenius norm and $s$-norm of a tensor\label{sec:Frobenius-norm-and-s-norm}}

Since we have no way to add/subtract two tensors in CTD form directly,
computing the Frobenius norm of the difference of two tensors via
\begin{equation}
\left\Vert \mathcal{U}-\mathcal{V}\right\Vert _{F}=\left(\left\Vert \mathcal{U}\right\Vert _{F}^{2}-2\langle\mathcal{U},\mathcal{V}\rangle+\left\Vert \mathcal{V}\right\Vert _{F}^{2}\right)^{1/2}\label{eq:FrobDiff-loss-digits}
\end{equation}
results in the loss of significant digits if the two tensors are close.
As an alternative to the Frobenius norm, we use the spectral, or \emph{$s$-norm,}
of a tensor by computing its rank-one approximation. Rank-one approximations
of a tensor are well understood (see, e.g., \cite{LA-MO-VA:2000,ZHA-GOL:2001,SIL-LIM:2008,HIL-LIM:2009}).
We use the fact that the largest $s$-value of the rank-one approximation
has all the necessary properties to be used as a norm (see Lemma below). 

The rank-one approximation of tensor $\mathcal{U}$ can be found by
solving the system of multi-linear equations \cite{LA-MO-VA:2000,ZHA-GOL:2001,LIM:2005,HIL-LIM:2009},
\begin{equation}
\sigma\mathbf{x}_{j'}=\sum_{l=1}^{r}\sigma_{l}\prod_{j=1,j\neq j'}^{d}\langle\mathbf{u}_{j}^{(l)},\mathbf{x}_{j}\rangle,\,\,\,\, j'=1,2,\dots,d,\label{eq:rank-one-fixed-point}
\end{equation}
where $\left\Vert {\bf x}_{1}\right\Vert _{F}=\left\Vert {\bf x}_{2}\right\Vert _{F}=\cdots=\left\Vert {\bf x}_{d}\right\Vert _{F}=1$.
The solution of this system maximizes 
\begin{equation}
\sigma=\sum_{l=1}^{r}\sigma_{l}\prod_{j=1}^{d}\langle\mathbf{u}_{j}^{(l)},\mathbf{x}_{j}\rangle,\label{eq:rank-one-fixed-point-svalue}
\end{equation}
so that $\sigma$ coincides with $s$-norm $\left\Vert \mathcal{U}\right\Vert _{s}$
if the global maximum is attained in solving (\ref{eq:rank-one-fixed-point}).
\begin{lem}
\label{thm:s-norm}Suppose $\mathcal{U}=\sum_{l=1}^{r}\sigma_{l}\bigotimes_{j=1}^{d}\mathbf{u}_{j}^{(l)}$
is a rank-$r$ CTD. Consider
\begin{equation}
\left\Vert \mathcal{U}\right\Vert _{s}=\sup_{\left\Vert \hat{\mathbf{x}}_{j}\right\Vert _{F}=1,\,\, j=1,\dots,d}\left(\sum_{l=1}^{r}\sigma_{l}\prod_{j=1}^{d}\langle\mathbf{u}_{j}^{(l)},\mathbf{\widehat{x}}_{j}\rangle\right),\label{eq:lemma-s-norm}
\end{equation}
where $\hat{\mathcal{X}}=\otimes_{j=1}^{d}\hat{\mathbf{x}}_{j}$ is
the best rank-one approximation of $\mathcal{U}$. Then $\left\Vert \mathcal{U}\right\Vert _{s}$
satisfies the properties of a norm,\end{lem}
\begin{enumerate}
\item $\left\Vert \alpha\mathcal{U}\right\Vert _{s}=\left|\alpha\right|\,\left\Vert \mathcal{U}\right\Vert _{s}$
for any $\alpha\in\mathbb{R}$
\item $\left\Vert \mathcal{U}\right\Vert _{s}=0$ if and only if $\mathcal{U}=0$
\item $\left\Vert \mathcal{U}+\mathcal{V}\right\Vert _{s}\leq\left\Vert \mathcal{U}\right\Vert _{s}+\left\Vert \mathcal{V}\right\Vert _{s}$\end{enumerate}
\begin{proof}
First, we observe that for a rank-one tensor $\mathcal{X}=\sigma\hat{\mathcal{X}}$
satisfying (\ref{eq:rank-one-fixed-point}) and (\ref{eq:rank-one-fixed-point-svalue}),
\begin{eqnarray}
\left\Vert \mathcal{U}-\mathcal{X}\right\Vert _{F}^{2} & = & \left\Vert \mathcal{U}\right\Vert _{F}^{2}-2\langle\mathcal{U},\mathcal{X}\rangle+\left\Vert \mathcal{X}\right\Vert _{F}^{2}\nonumber \\
 & = & \left\Vert \mathcal{U}\right\Vert _{F}^{2}-\sigma^{2}.\label{eq: rank one variational error}
\end{eqnarray}
Thus, the best rank-one approximation with unit norm, $\hat{\mathcal{X}}$,
maximizes the quantity $\langle\mathcal{U},\hat{\mathcal{X}}\rangle$. 

If $\alpha\geq0$, then (1) is obvious. If $\alpha<0$, then $\langle\mathcal{U},\hat{\mathcal{X}}\rangle=-\alpha\sigma$
and changing the sign of any single $\hat{\mathbf{x}}_{j}$, $j=1,\dots,d$,
we obtain (1). For any $\mathcal{U}\neq0$, there exists at least
one rank-one tensor $\mathcal{X}$ such that $\left\langle \mathcal{U},\mathcal{X}\right\rangle >0$,
implying that $\left\Vert \mathcal{U}\right\Vert _{s}>0$. In fact,
we can take $\mathcal{X}=\mathcal{U}^{(l')}$, where $\mathcal{U}^{(l')}$
is one of the rank-one terms of $\mathcal{U}$. Indeed, assuming that
$\left\langle \mathcal{U},\mathcal{U}^{(l)}\right\rangle \leq0$ for
all $l=1,\dots,r_{u}$, and writing the Frobenius norm of $\mathcal{U}$
as 
\begin{equation}
\left\Vert \mathcal{U}\right\Vert _{F}^{2}=\sum_{l}\sigma_{l}\langle\mathcal{U},\mathcal{U}^{(l)}\rangle\leq0,\label{eq: frob norm leq 0?}
\end{equation}
we conclude that $\mathcal{U}=0$. Thus, there is at least one rank-one
term of $\mathcal{U}$ such that $\left\langle \mathcal{U},\mathcal{U}^{(l)}\right\rangle >0$.
Finally, the triangle inequality follows since
\begin{eqnarray}
\left\Vert \mathcal{U}+\mathcal{V}\right\Vert _{s} & = & \sup_{\mathcal{X}:\left\Vert \mathcal{X}\right\Vert _{F}=1}\left\langle \mathcal{U}+\mathcal{V},\mathcal{X}\right\rangle \nonumber \\
 & \leq & \sup_{\mathcal{Y}:\left\Vert \mathcal{Y}\right\Vert _{F}=1}\left\langle \mathcal{U},\mathcal{Y}\right\rangle +\sup_{\mathcal{Z}:\left\Vert \mathcal{Z}\right\Vert _{F}=1}\left\langle \mathcal{V},\mathcal{Z}\right\rangle \nonumber \\
 & = & \left\Vert \mathcal{U}\right\Vert _{s}+\left\Vert \mathcal{V}\right\Vert _{s}.\label{eq: s norm triangle ineq}
\end{eqnarray}
where $\mathcal{X}$, $\mathcal{Y}$ and $\mathcal{Z}$ are rank-one
tensors.
\end{proof}
The equations (\ref{eq:rank-one-fixed-point}) are solved by an iteration
equivalent to ALS for approximating via rank-one tensors. However,
in this case, there is no linear system to solve and, thus, the iteration
resembles the power method for matrices. Specifically, for each direction
$j$, the iteration proceeds by updating the left hand side, vector
$\mathbf{x}_{j'}$, by evaluating the right hand side with the currently
available vectors $\mathbf{x}_{j}$. Re-normalizing $\mathbf{x}_{j'}$
to obtain the normalization factor $\sigma$, and sweeping through
the directions, the iteration terminates when the change in $\sigma$
is small. However, unlike the power method for matrices, this iteration
may have more than one stationary point, meaning that the answer may
depend on the initialization. 
\begin{rem}
Although the definition of the\emph{ $s$-norm} parallels the matrix
2-norm and can be useful as a way of estimating errors, computing
the $s$-norm exactly for arbitrary dense tensors is claimed to be
an NP-hard problem in \cite{HIL-LIM:2009}. We note that for symmetric
tensors, the global convergence of the iteration described above has
been claimed in \cite{KOL-MAY:2011}.
\end{rem}
We discuss our approach to initialization below but first consider
the cost of estimating the $s$-norm and its relation to the Frobenius
norm. Let us assume that $M_{j}=M$ for $j=1,\dots,d$, and let $n_{it}$
denote the number of iterations required for the rank one iteration
to converge ($n_{it}$ is usually small). For each direction, $(d-1)\cdot r$
inner products are computed at a cost of $\mathcal{O}(M)$ operations
each. Since only one inner product must be updated at a time, the
total computational cost is estimated as 
\begin{equation}
n_{it}\cdot d\cdot r\cdot M.\label{eq:rank-one-als-cost}
\end{equation}
For comparison, the cost of computing the Frobenius norm via (\ref{eq:intro-frobenius-inner-product-for-CTD})
and (\ref{eq:intro-frob-norm}) is $\mathcal{O}\left(d\cdot r^{2}\cdot M\right)$,
so that estimating the $s$-norm is faster if $n_{it}<r$. Another
advantage of using the $s$-norm is that there is no loss of significant
digits in computing it via (\ref{eq:rank-one-fixed-point}). We have
\begin{lem}
\textup{\label{thm:s-equivalence-1}}The Frobenius and $s$-norms
satisfy\textup{ 
\begin{equation}
\frac{1}{\sum_{l=1}^{r}\sigma_{l}}\left\Vert \mathcal{U}\right\Vert _{F}^{2}\leq\left\Vert \mathcal{U}\right\Vert _{s}\leq\left\Vert U\right\Vert _{F},\label{eq: snorm inequalities}
\end{equation}
where $\mathcal{U}$ is defined in (\ref{eq:CTD}), $U$ in (\ref{eq:umat})
and $\left\Vert U\right\Vert _{F}=\left(\mbox{tr}U^{*}U\right)^{1/2}=\left(\sum_{l=1}^{r}\sigma_{l}^{2}\right)^{1/2}$.}\end{lem}
\begin{proof}
By definition of the matrix $2$-norm, we have
\begin{equation}
\left\Vert U\right\Vert _{2}=\left\Vert U^{*}\right\Vert _{2}=\max_{\left\Vert \mathbf{x}\right\Vert _{2}=1}\left\Vert U^{*}\mathbf{x}\right\Vert _{2}.\label{eq:2-norm-defn-1}
\end{equation}
Here $\mathbf{x}$ is a vector in $\mathbb{R}^{N}$ corresponding
to a dense tensor. Thus, taking $\mathbf{x}$ corresponding to the
rank-one tensor achieving the best approximation to $\mathcal{U}$
and satisfying (\ref{eq:rank-one-fixed-point}) and (\ref{eq:rank-one-fixed-point-svalue}),
we obtain $\left\Vert \mathcal{U}\right\Vert _{s}\leq\left\Vert U\right\Vert _{2}\le\left\Vert U\right\Vert _{F}$.
Starting from $\left\Vert \mathcal{U}\right\Vert _{F}^{2}=\sum_{l=1}^{r}\sigma_{l}\langle\mathcal{U}^{(l)},\mathcal{U}\rangle$
and replacing the rank one terms $\mathcal{U}^{(l)}$ with the best
rank one approximation $\mathcal{X}$ of the tensor $\mathcal{U}$
, we obtain
\begin{equation}
\sum_{l=1}^{r}\sigma_{l}\langle\mathcal{U}^{(l)},\mathcal{U}\rangle\le\sum_{l=1}^{r}\sigma_{l}\langle\mathcal{X},\mathcal{U}\rangle=\left\Vert \mathcal{U}\right\Vert _{s}\sum_{l=1}^{r}\sigma_{l}.\label{eq: lemma proof inequality}
\end{equation}

\end{proof}

\subsection{Approximating the $s$-norm}

The proof of Lemma~\ref{thm:s-norm} assumes that the best rank-one
approximation can be obtained. However, in general, the iteration
in (\ref{eq:rank-one-fixed-point}) may converge to a local maximum
and, thus, the result may depend on the initialization. Therefore,
we need a systematic approach to initializing (\ref{eq:rank-one-fixed-point}).
Several approaches to initialization have been suggested previously
(see e.g., \cite{LA-MO-VA:2000,KOF-REG:2001}). 

We initialize the iteration by using components of the given tensor
$\mathcal{U}$ (rather than a random initialization). In applications
we are interested in, a small number of terms typically dominate the
representation so that it often sufficient to initialize using the
term with the largest $s$-value. Alternatively, we generated the
initial guess using component matrices in each direction by either
averaging their columns or computing the singular vector corresponding
to their largest singular value. Obviously, these are heuristic choices
but they appear to work well in the computational environment we are
interested in. For any of these initialization methods, there is no
guarantee that they provide the globally optimal rank-one approximation
for an arbitrary CTD. 

In the environment of computing the tensor ID, a weaker form of the
definition for the $s$-norm may be appropriate. Instead of demanding
that the s-norm be the global maximum of (\ref{eq:rank-one-fixed-point-svalue}),
we can define it to be the maximally computed $\sigma$ using a specific
initialization method. However, this is only permitted if this definition
of the s-norm satisfies the triangle inequality in Lemma~\ref{thm:s-norm}.
In other words, if we denote $\sigma^{u}$ and $\sigma^{v}$ to be
the maximal computed (not necessarily global) $s$-values for tensors
$\mathcal{U}$ and $\mathcal{V}$, it is necessary that $\sigma^{(u+v)}$
computed using the same initialization method satisfies $\sigma^{u}+\sigma^{v}\leq\sigma^{(u+v)}$.
When computing tensor ID under this definition, we can always verify
the triangle inequality \textit{a posteriori}.

\section{Examples\label{sec:Examples}}

All of the numerical examples were implemented in MATLAB \cite{MATLAB:2012}.
We used the CTD data structure and basic routines available through
the Sandia Tensor Toolbox~2.5 \cite{TENSOR:2012} and implemented
ALS, tensor ID, and all additional routines with no special effort
to optimize or parallelize the codes. All experiments were performed
on a PC laptop with a 2.20 GHz Intel i7 chipset and 8 GB of RAM.

\subsection{Comparison of matrices associated with tensor ID algorithms\label{sub:Example-of-spectral-properties}}

We first illustrate the loss of significant digits in constructing
the tensor ID using the Gram matrix in (\ref{eq:intro-gmat}) by comparing
its numerical rank with that of the matrix generated via the randomized
approach in Section~\ref{sub:A-method-based-on-canonical-random-tensors}.

For the comparison, we construct a tensor $\mathcal{U}$ with terms
that are nearly orthogonal and with $s$-values decaying exponentially
fast. In this case, a tensor ID for a user-selected accuracy $\epsilon$
can be obtained by simply dropping terms with small $s$-values. Therefore,
we can directly compare the number of the skeleton terms simply by
estimating the numerical rank of the associated matrices (without
actually computing new tensor ID coefficients).

Consider a random tensor in dimension $d=20$, with $M_{1}=\cdots=M_{20}=128$
and $r=100$, and generated as
\begin{equation}
\mathcal{U}=\sum_{l=1}^{r}\sigma_{l}\mathcal{U}^{(l)}\,\,\,\,\,\,\,\,\,\,\,\,\,\,\,\,\mbox{\mbox{with}\,\,\,\,\,\,\,\,\,\,\,\,\,\,\,\,}\mathcal{U}^{(l)}=\bigotimes_{j=1}^{d}\mathbf{u}_{j}^{(l)},\,\,\,\,\,\,\,\,\,\,\,\,\,\, u_{i_{j}}^{(l)}\sim N(0,1),\label{eq:example-accuracy-random-tensor}
\end{equation}
where $N(0,1)$ denotes the normal distribution with zero mean and
unit variance and where the vectors $\mathbf{u}_{j}^{(l)}$ are normalized
to have unit Frobenius norm. The $s$-values assigned to the terms
are exponentially decaying,
\begin{equation}
\sigma_{l}=\exp\left(-l/2\right),\,\,\,\,\,\,\,\,\,\,\,\,\,\,\, l=1,\dots,r.\label{eq: example - exp decaying s values}
\end{equation}
By construction, the terms of $\mathcal{U}$ are nearly orthogonal
so that the truncation error incurred by removing small terms is approximately
\begin{equation}
\epsilon_{l'}=\left(\sum_{l>l'}\sigma_{l}^{2}\right)^{1/2}.\label{eq:spectrum-example-s-value-decay}
\end{equation}
Therefore, using $\epsilon_{l'}\leq\epsilon$ , the tensor ID for
accuracy $\epsilon$ should select the first $l'$ terms. In order
for the tensor ID to succeed in choosing these terms, the matrices
for its construction must have numerical rank greater than $l'$.

We compute the Gram matrix $G$ via (\ref{eq:intro-gmat}). For the
random projection method, we generate the tensors $\mathcal{R}^{(l)}$
for $l=1,\dots,r$ and form the $r\times r$ projection matrix $Y$
via (\ref{eq:intro-ymat-rproj}). We then compute the singular values
of the matrices $G$ and $Y$.

The singular values of the matrices $G$ and $Y$ are shown in Figure~\ref{fig:ex1-spectra}.
For reference, we also plot the $s$-values of $\mathcal{U}$. Notice
that for accuracy $\epsilon\approx10^{-16}$ the numerical rank of
$G$ is only $\sim35$ since, as expected, the singular values of
$G=U^{*}U$ decay twice as fast logarithmically as those of matrix
$U$. This implies that the tensor ID using $G$, computed in double
precision with $\approx16$ accurate digits, loses its ability to
distinguish significant terms for requested accuracies smaller than
$\approx10^{-8}$. On the other hand, for accuracies $\ll10^{-8}$,
the numerical rank of $Y$ allows us to select for up to $75-80$
terms. The displayed results do not depend in any significant way
on the choices of distributions in the random projection method described
in Section~\ref{sub:A-method-based-on-canonical-random-tensors}. 

\begin{figure}
\begin{centering}
\includegraphics[scale=0.5]{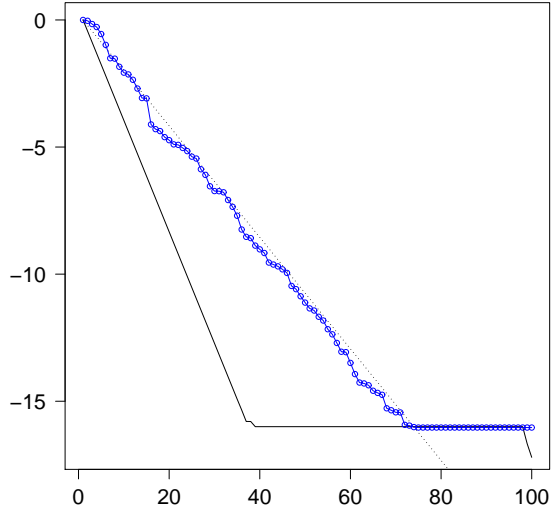}
\par\end{centering}

\caption{\label{fig:ex1-spectra}Decay of singular values of the matrices in
Example~\ref{sub:Example-of-spectral-properties}. The logarithm
($\log_{10}$) of the singular values are displayed as a function
of their index. The singular values of $G$ are plotted with a solid
line. The singular values of $Y$ are displayed using symbol ``$\circ$'',
while the original $s$-values of the tensor $\mathcal{U}$ are displayed
using fine dots and nearly coincide with the singular values of $Y$.}

\end{figure}

Continuing with this example, we impose additional structure on the
terms of $\mathcal{U}$ and choose the last $30$ terms (at random)
from the first $70$ terms and give them the exponentially decaying
weights in (\ref{eq:spectrum-example-s-value-decay}). Since by the
original construction the terms were nearly mutually orthogonal, for
double precision accuracy the algorithms should produce $\sim70$
terms (cf. Figure~\ref{fig:example-basic-reduction-via-tensor-ID}),
i.e., choose all the linearly independent terms of the tensor.

Results for the Gram and randomized tensor ID algorithms are shown
in Figure~\ref{fig:example-basic-reduction-via-tensor-ID}. The left
plot shows the relative $s$-norm error plotted against the separation
rank, $\ell$, of the tensor IDs computed via the matrices $G$ and
$Y$. The right plot shows the separation rank of the tensor ID approximation
as a function of $\ell$. The underlying separation rank of the CTD
is known to be $r\sim70$ for $\epsilon\sim\epsilon_{machine}$, and
the error for the randomized methods levels off when $k$ approaches
this value. The Gram method, however, is only able to identify the
first $\sim35$ terms due to inherent loss of accuracy. 

The Frobenius error for the randomized tensor ID approximately matches
the $s$-norm error until the cutoff of $\sqrt{\epsilon_{machine}}$
is attained, at which point it stays constant with respect to $k$
(not shown). Hence, this example also demonstrates the usefulness
of the $s$-norm when high accuracy is sought.

\begin{figure}
\begin{centering}
\includegraphics[scale=0.4]{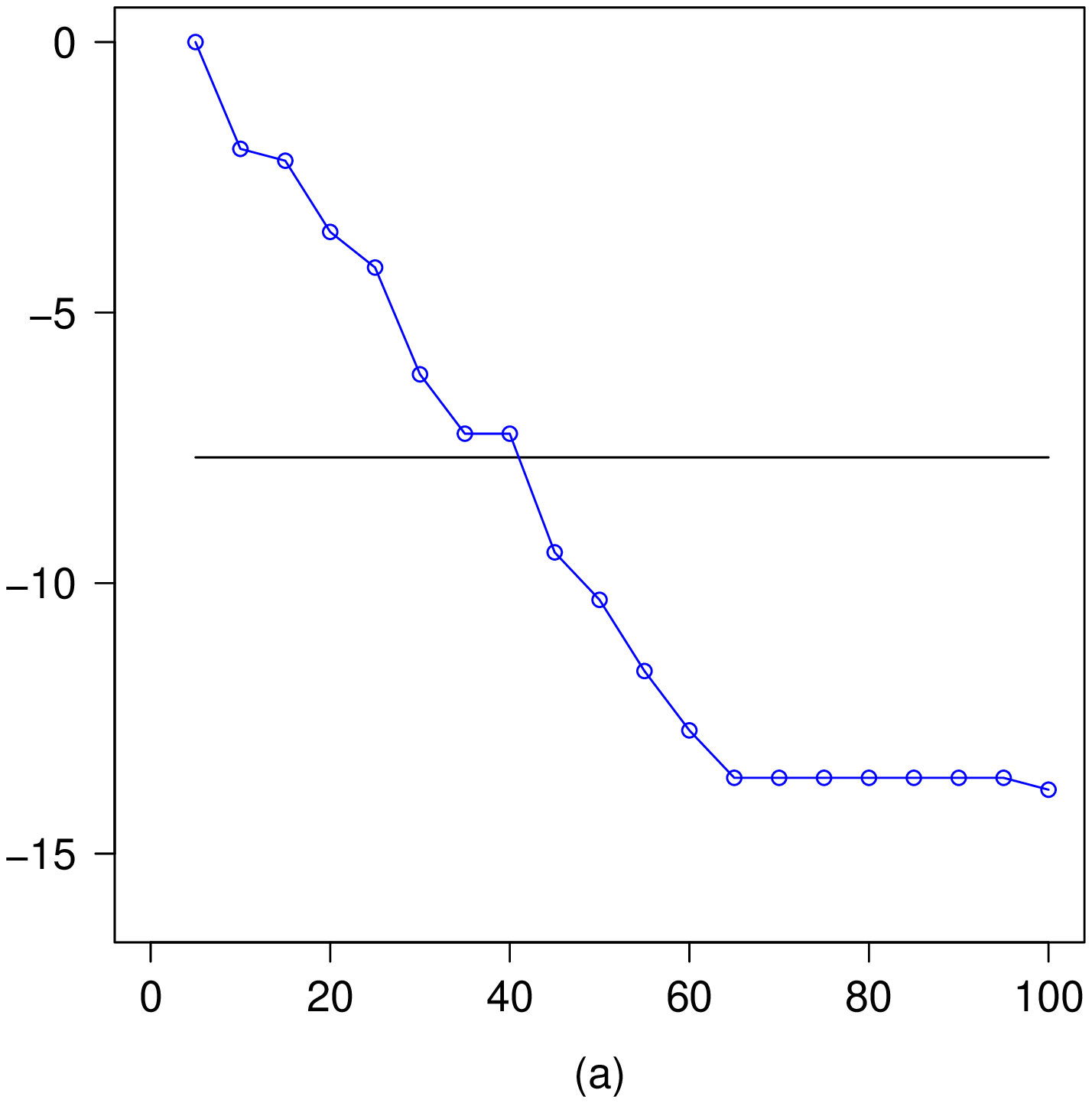}\includegraphics[scale=0.4]{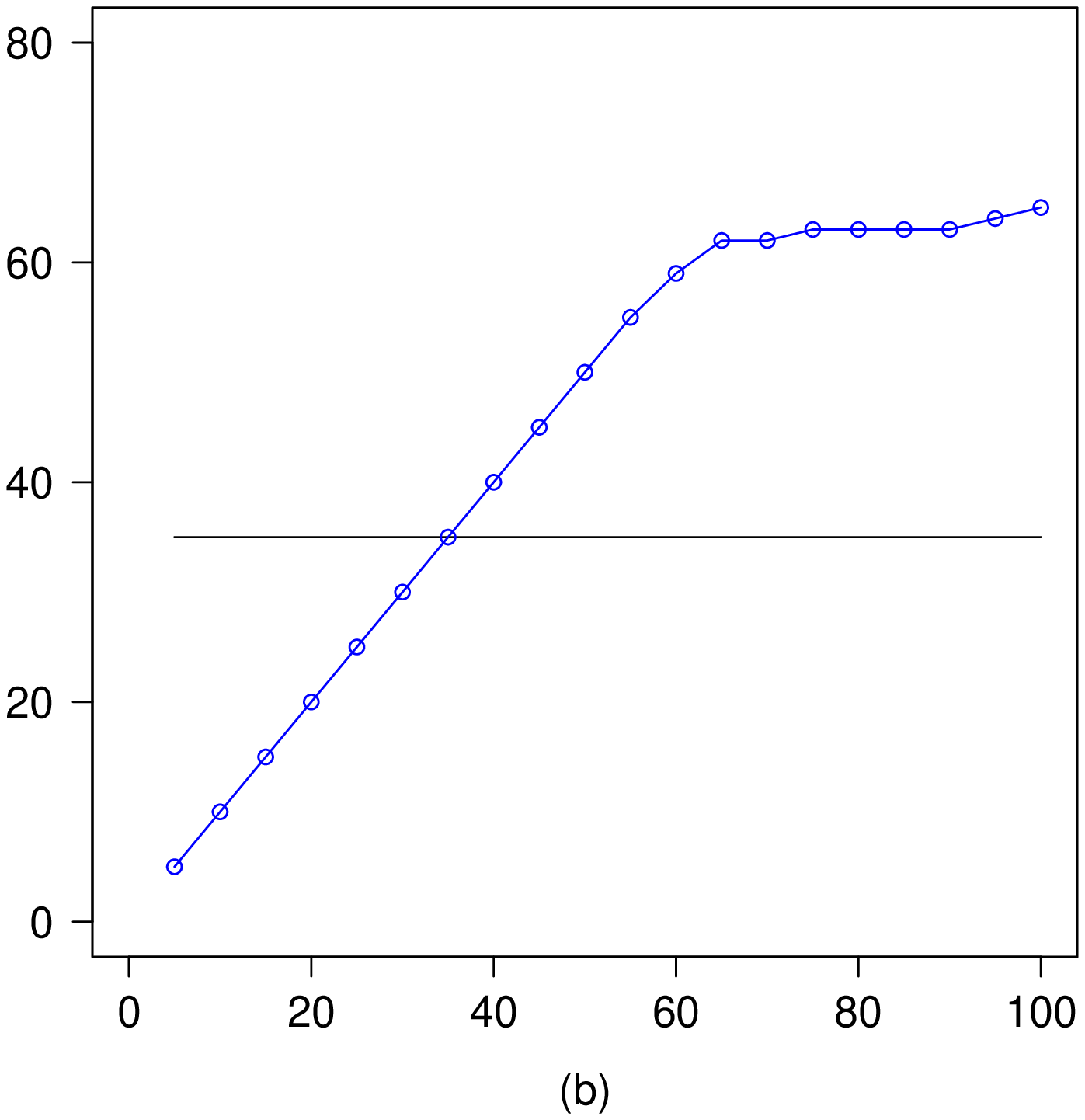}
\par\end{centering}

\caption{\label{fig:example-basic-reduction-via-tensor-ID}(a) Logarithm ($\mbox{\ensuremath{\log}}_{10}$)
of relative accuracy of computing tensor ID using the random projection
method and (b) the resulting separation rank as a function of the
number of projections, $\ell$, for Example~\ref{sub:Example-of-spectral-properties}.
Results for the Gram matrix approach (which do not depend on the parameter
$\ell$) are shown as a horizontal solid line. The errors are computed
using the $s$-norm. }
\end{figure}

\subsection{The tensor ID within convergent, self-correcting Schulz iteration}

We present three examples of using the tensor ID within the SGTI approach
in order to accelerate reduction of separation rank. These examples
were originally presented in \cite{BIAGIO:2012}. We consider the
quadratically convergent, self-correcting Schulz iteration \cite{SCHULZ:1933},
given by
\begin{eqnarray}
\mathbb{X}_{n+1} & = & 2\mathbb{X}_{n}-\mathbb{X}_{n}\mathbb{B}\mathbb{X}_{n},\nonumber \\
\mathbb{X}_{0} & = & \alpha\mathbb{B}^{*},\label{eq: schulz iteration-1}
\end{eqnarray}
where $\alpha$ is chosen so that the initial error $\mathbb{E}_{0}$
satisfies $\left\Vert \mathbb{E}_{0}\right\Vert =\left\Vert \mathbb{I}-\mathbb{X}_{0}\mathbb{B}\right\Vert <1$.
In these examples, the operator $\mathbb{B}$ is a preconditioned
elliptic operator whose inverse corresponds to the Green's function
of a Poisson equation (see more details below). Within each iteration,
we first form the quantity $2\mathbb{I}-\mathbb{BX}_{n}$ which we
then left-multiply by $\mathbb{X}_{n}$ to obtain $\mathbb{X}_{n+1}$.
Both of these operations significantly increase the separation rank
and require a reduction step. 

The reduction step, which would typically be performed using ALS,
is instead performed with the randomized projection tensor ID Algorithm~\ref{alg:Tensor-ID-via-random-projections}.
Only after each complete iteration, when the separation rank has been
reduced as much as possible via the tensor ID, is ALS invoked to further
refine the approximation. In doing so, we avoid using ALS in the usual
manner, i.e., to achieve a certain accuracy of approximation. Instead,
its role is limited to reducing the dynamic range of the $s$-values
(i.e., avoiding near cancellation of terms with large $s$-values)
by running the algorithm for only a fixed (small) number of iterations.
The reduction errors of the tensor ID and several ALS iterations are
then corrected by the next Schulz iteration.

\subsubsection{Inverse operator for the Poisson equation\label{ex:Inverse-operator-for-poisson}}

As the first example, we consider the periodic, constant coefficient
Poisson equation. In this case we know that the Green's function has
an efficient separated representation, see e.g., \cite{B-F-H-K-M:2012},
and we want to demonstrate that we can approximate the Green's function
starting with the differential operators in 
\begin{eqnarray}
-\Delta u(\mathbf{x}) & = & f(\mathbf{x}),\,\,\,\,\,\,\,\,\,\,\,\,\mathbf{x}\in(0,1)^{3},\label{eq:periodic constant coefs}\\
u(0,y,z) & = & u(0,y,z),\nonumber \\
u(x,0,z) & = & u(x,1,z),\nonumber \\
u(x,y,0) & = & u(x,y,1).\nonumber 
\end{eqnarray}
where 
\[
\Delta=\frac{\partial^{2}}{\partial x_{1}^{2}}+\frac{\partial^{2}}{\partial x_{2}^{2}}+\frac{\partial^{2}}{\partial x_{3}^{2}}.
\]
We use eighth-order finite differences to discretize the second derivative
in each direction to obtain the $512\times512$ matrix $A$, leading
to the operator with separation rank $r=3$,
\[
\mathbb{A}=A\otimes I\otimes I+I\otimes A\otimes I+I\otimes I\otimes A
\]
where $I$ denotes the identity matrix. We represent the matrix $A$
in a wavelet basis to ensure that both the operator and its inverse
are sparse \cite{BE-CO-RO:1991,BE-CO-RO:1992,BEYLKI:1992,BEYLKI:1994}.
In a wavelet basis, the second derivative operator has a diagonal
preconditioner, $P$ (see e.g., \cite{BEYLKI:1994}), such that the
condition number of $PAP$ is $\mathcal{O}(1)$. Applying such preconditioner
in dimension $d=3$ results in the well-conditioned operator
\[
\mathbb{B}=\left(PAP\right)\otimes I\otimes I+I\otimes\left(PAP\right)\otimes I+I\otimes I\otimes\left(PAP\right).
\]
Since the problem (\ref{eq:periodic constant coefs}) is periodic,
the operators $\mathbb{A}$ and $\mathbb{B}$ have a one-dimensional
null space spanned by a constant. This necessitates the use of the
one dimensional projector within the Schulz iteration in order to
avoid accumulation of the error in the null space. The results of
this computation are shown in Figure~\ref{fig:-Schulz-iteration-constant coefs}. 

\begin{figure}
\begin{centering}
\includegraphics[scale=0.4]{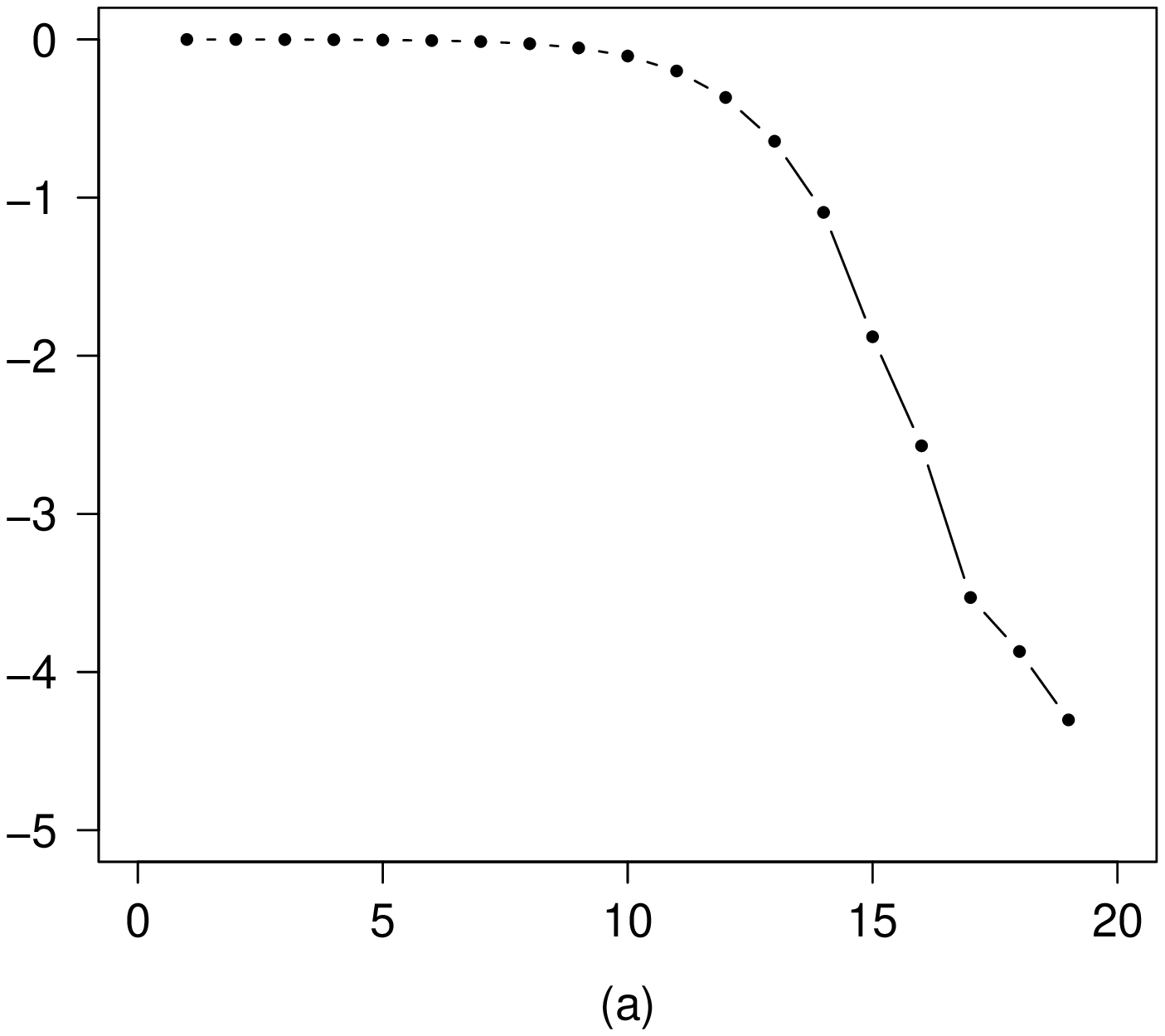}\includegraphics[scale=0.4]{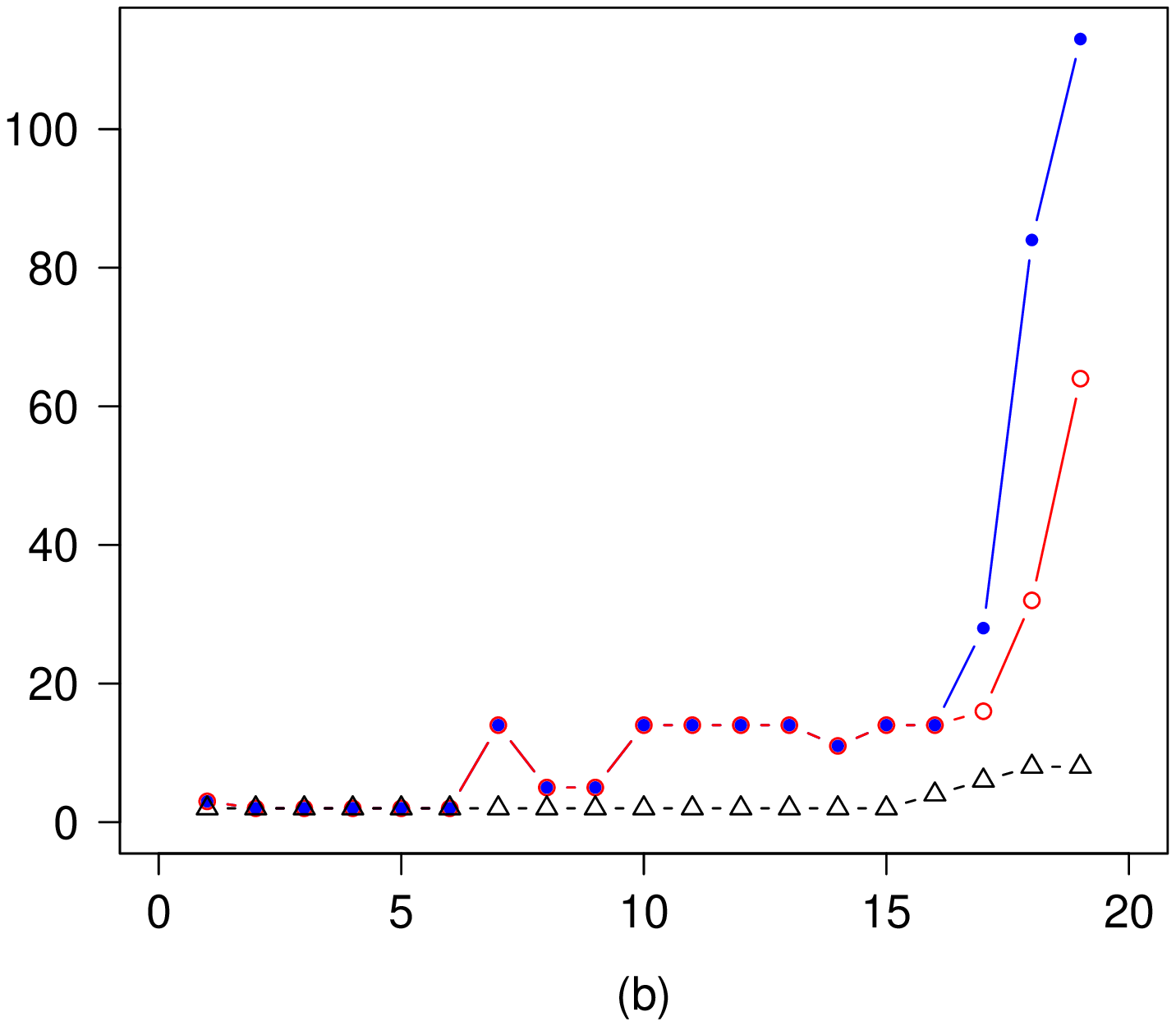}
\par\end{centering}

\caption{\label{fig:-Schulz-iteration-constant coefs}Results for Example~\ref{ex:Inverse-operator-for-poisson}.
Schulz error $\left\Vert \mathbb{E}_{n}\right\Vert =\frac{1}{2}\left(\left\Vert \mathbb{I}-\mathbb{X}_{n}\mathbb{B}\right\Vert +\left\Vert \mathbb{I}-\mathbb{B}\mathbb{X}_{n}\right\Vert \right)/\left\Vert I\right\Vert $
per iteration $n=1,2,3,\dots$ for constructing the Green's function
for (\ref{eq:periodic constant coefs}) and separation rank of Schulz
iterate $\mathbb{X}_{n}$ before (dots) and after (circles) reduction
by tensor ID. The triangles show the separation rank after applying
ALS (with a fixed number of iterations) to the result of tensor ID.}

\end{figure}

\subsubsection{Variable coefficient elliptic operator in dimension $d=10$.\label{ex:Variable-coefficient-elliptic-operator}}

Next we consider the 10-dimensional PDE on the unit cube,
\begin{eqnarray}
-\nabla\cdot(a(\mathbf{x})\nabla u(\mathbf{x})) & = & f(\mathbf{x}),\,\,\,\,\,\,\,\,\,\,\,\,\mathbf{x}\in(0,1)^{10},\label{eq:elliptic PDE var coefs}
\end{eqnarray}
with periodic boundary conditions and where
\begin{equation}
a(\mathbf{x})=1-0.9\exp\left(-3\times10^{3}\cdot(\mathbf{x}-0.5)^{2}\right).\label{eq:poisson var coef gaussian}
\end{equation}
In this case we reformulate the problem by using the constant coefficient
Green's function to convert (\ref{eq:elliptic PDE var coefs}) into
an integral equation, i.e., the constant coefficient Green's function
is used as a preconditioner. By separating the constant from the variable
terms in (\ref{eq:elliptic PDE var coefs}), the discretized elliptic
operator $\mathbb{A}$ can be split into constant and variable parts,
\begin{equation}
\mathbb{A}=\mathbb{A}_{c}+\mathbb{A}_{v}.\label{eq: pde operator space + stoch-1}
\end{equation}
Applying the discretized, constant coefficient Green's function $\mathbb{G}_{c}$
to $\mathbb{A}$, we obtain
\begin{equation}
\mathbb{B}\,=\,\mathbb{G}_{c}\left(\mathbb{A}_{c}+\mathbb{A}_{v}\right)\,\,=\,\,\mathbb{I}+\mathbb{G}_{c}\mathbb{A}_{v}.\label{eq: operator B-1}
\end{equation}
The variable coefficient Green's function $\mathbb{G}$ of equation
(\ref{eq:elliptic PDE var coefs}) can now be constructed by computing
$\mathbb{B}^{-1}$ via Schulz iteration and setting $\mathbb{G}=\mathbb{B}^{-1}\mathbb{G}_{c}$. 

In this example, second order staggered finite differences are used
to discretize the derivative operators in each direction at $128$
equispaced points, leading to the elliptical operator $\mathbb{A}$
with separation rank $20$. Upon applying the constant coefficient
Green's function $\mathbb{G}_{c}$ and truncating terms with small
$s$-values, the preconditioned operator $\mathbb{B}$ has nominal
separation rank $271$. To $6$ digits of relative accuracy, however,
the separation rank of $\mathbb{B}$ may be dramatically reduced (via,
e.g., ALS iteration), and thus the iteration proceeds on an operator
of separation rank only $5$. As before, we represent all operators
in a wavelet basis to ensure the sparsity of both $\mathbb{B}$ and
$\mathbb{B}^{-1}$. Results are shown in Figure~\ref{fig:Schulz-iteraiton-var-coefs}.

\begin{figure}
\begin{centering}
\includegraphics[scale=0.4]{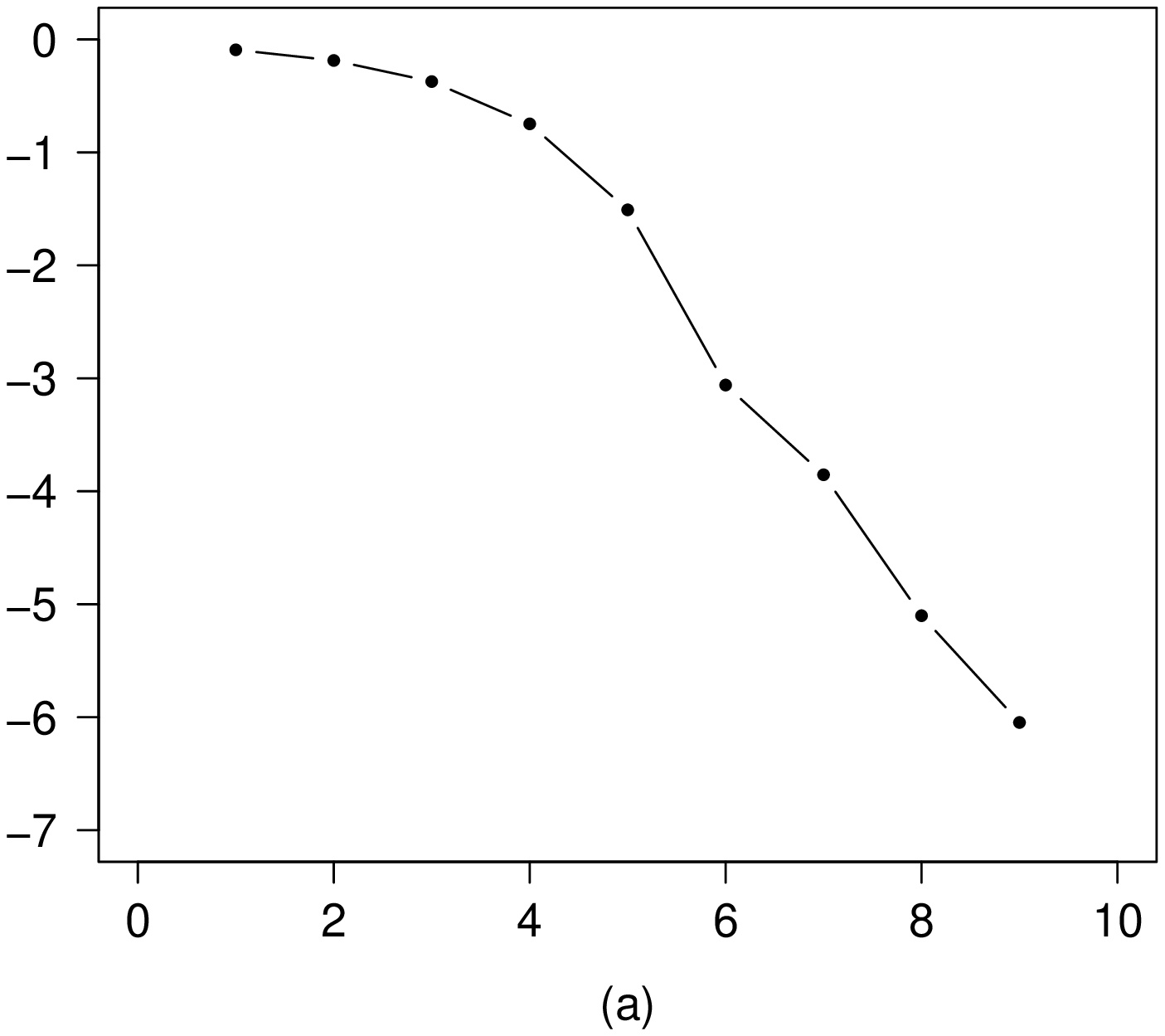}\includegraphics[scale=0.4]{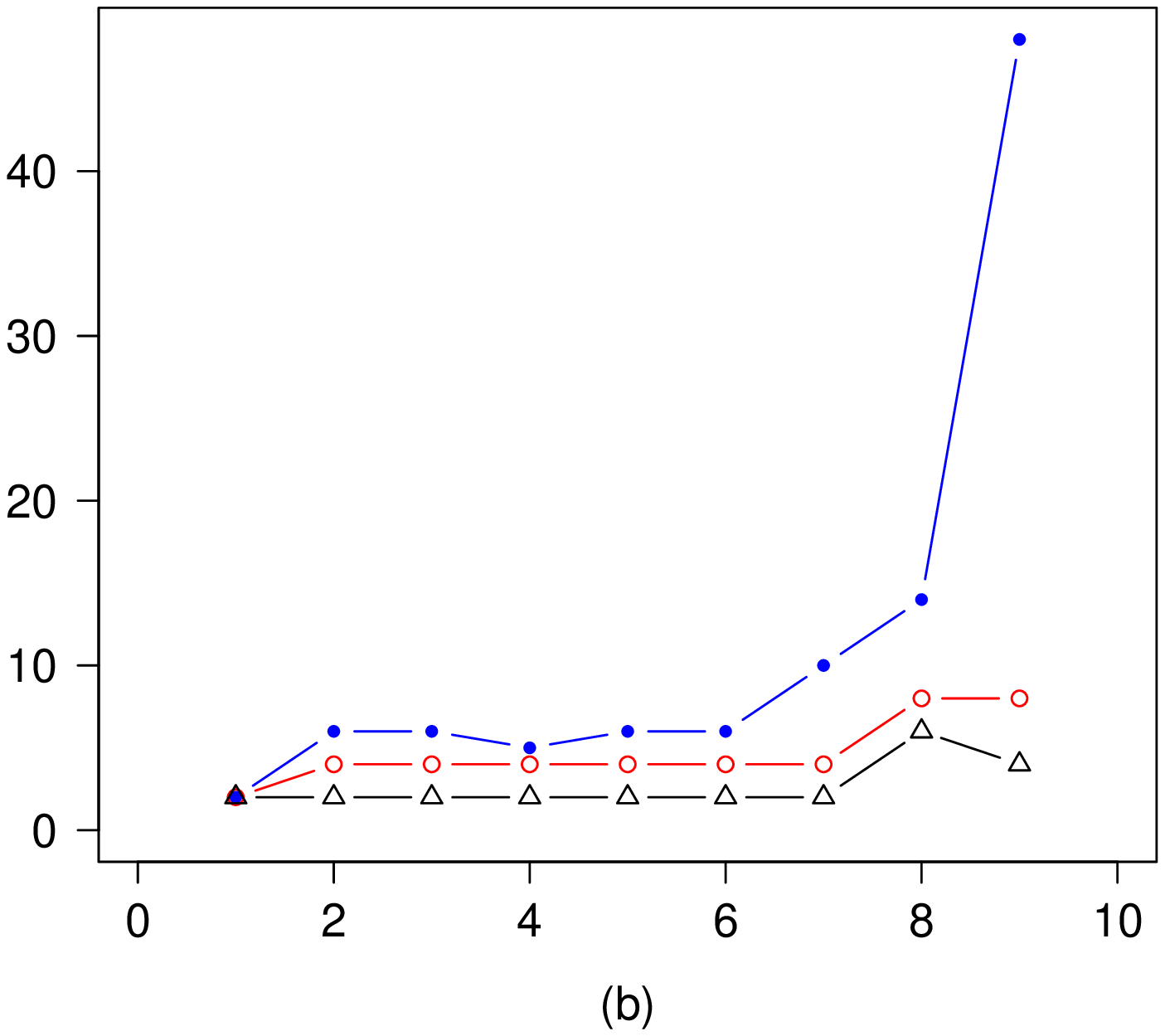}
\par\end{centering}

\caption{\label{fig:Schulz-iteraiton-var-coefs}Results for Example~\ref{ex:Variable-coefficient-elliptic-operator}.
Schulz error $\left\Vert \mathbb{E}_{n}\right\Vert =\frac{1}{2}\left(\left\Vert \mathbb{I}-\mathbb{X}_{n}\mathbb{B}\right\Vert +\left\Vert \mathbb{I}-\mathbb{B}\mathbb{X}_{n}\right\Vert \right)/\left\Vert I\right\Vert $
per iteration $n=1,2,3,\dots$ on $\log_{10}$ scale and separation
rank of Schulz iterate before reduction (dots) and after (circles)
tensor ID. The triangles show the separation rank after applying ALS
(with a fixed number of iterations) to the result of tensor ID.}

\end{figure}

\subsubsection{Stochastic PDE in dimension $d=8$.\label{sub:Stochastic-PDE-in-8d}}

In our last example, we consider the 8-dimensional stochastic PDE,
\[
-\nabla\cdot(a(\mathbf{x},\omega)\nabla u(\mathbf{x},\omega))=f(\mathbf{x},\omega),\,\,\,\,\,\,\,\,\,\,\mathbf{x}\in(0,1)^{3},
\]
where $\omega\in\Omega$ corresponds to probability space $(\Omega,\mathcal{F},P)$
and the (spatially asymmetric) variable coefficient is given by 
\[
a(\mathbf{x},\omega)=1+\sum_{l=1}^{5}2^{-l}a_{l}(\omega)\sin(2l\pi x)\sin(2l\pi y)\sin(2(l+1)\pi z).
\]
In other words, the variable coefficient is understood to have a deterministic
spatial part with random coefficients, which we take in this example
to be uniformly distributed, $a_{l}\sim U\left(\left[-1,1\right]\right).$
The function $a$ may be thought of as a Karhunen-Loeve (KL) expansion
of a random field with some (here, unspecified) covariance function.
The resulting operator is thus 8-dimensional, with three spatial and
five stochastic dimensions, the latter of which are discretized at
Clenshaw-Curtis quadrature nodes,
\[
a_{l}(\omega_{n})=-\cos\left(\frac{\pi m}{M_{stoch}-1}\right),\,\,\,\,\,\,\, l=1,\dots,5,\,\,\,\,\,\,\, m=1,\dots,M_{stoch}-1.
\]
As before, the spatial operator is discretized on a staggered grid
at $128$ points using second order finite differences, and we use
$M_{stoch}=16$ nodes in the stochastic directions. For the preconditioner,
we use the constant coefficient Green's function in the spatial directions
and identity matrices in the stochastic directions, leading to a preconditined
operator $\mathbb{B}$ with nominal separation rank of $1188$. This
number is reduced to $24$ by truncating terms with small $s$-svalues
and then applying ALS to the result. Results are shown in Figure~\ref{fig:Schulz-iteration-stochastic-pde}.

\begin{figure}
\begin{centering}
\includegraphics[scale=0.4]{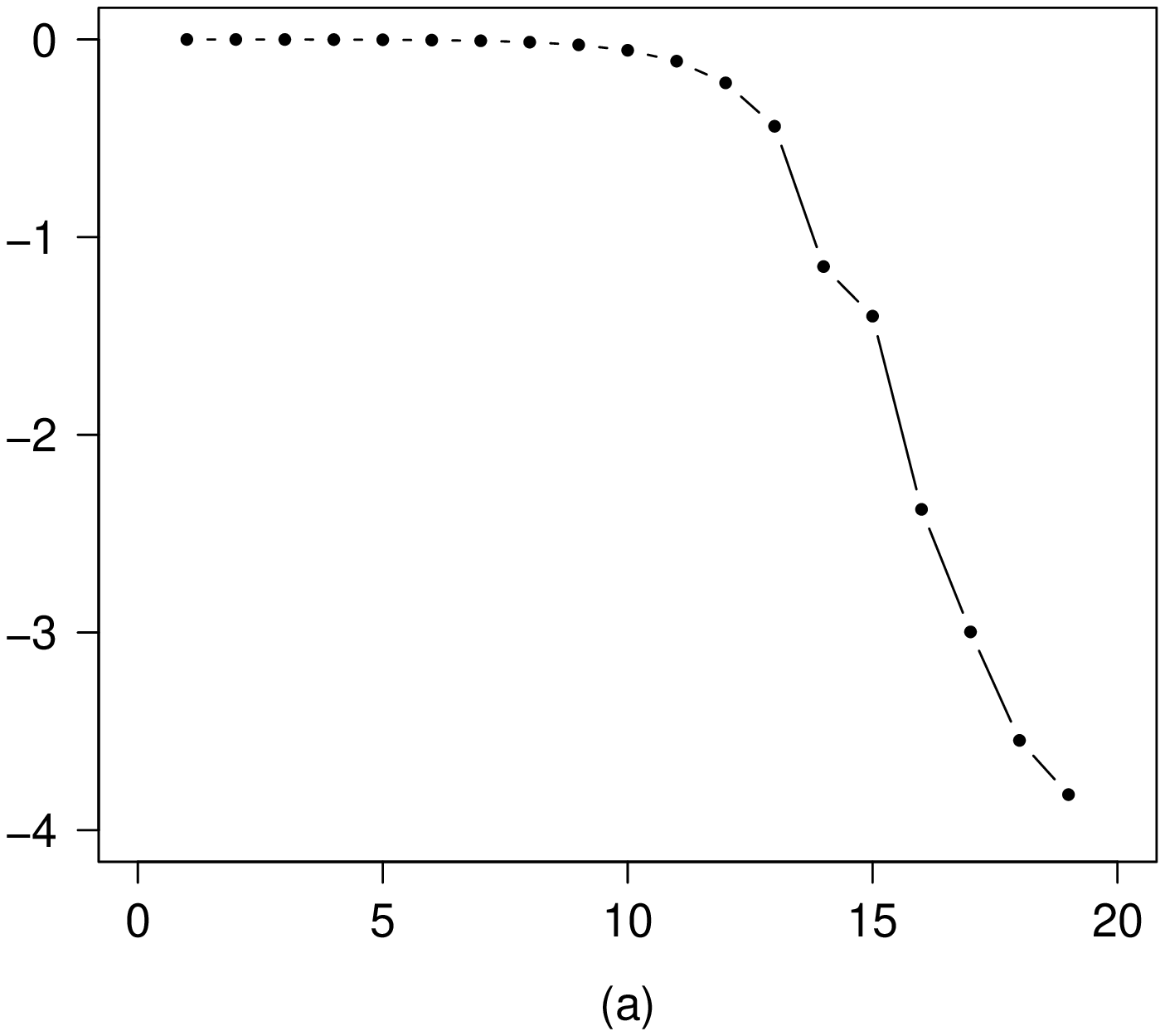}\includegraphics[scale=0.4]{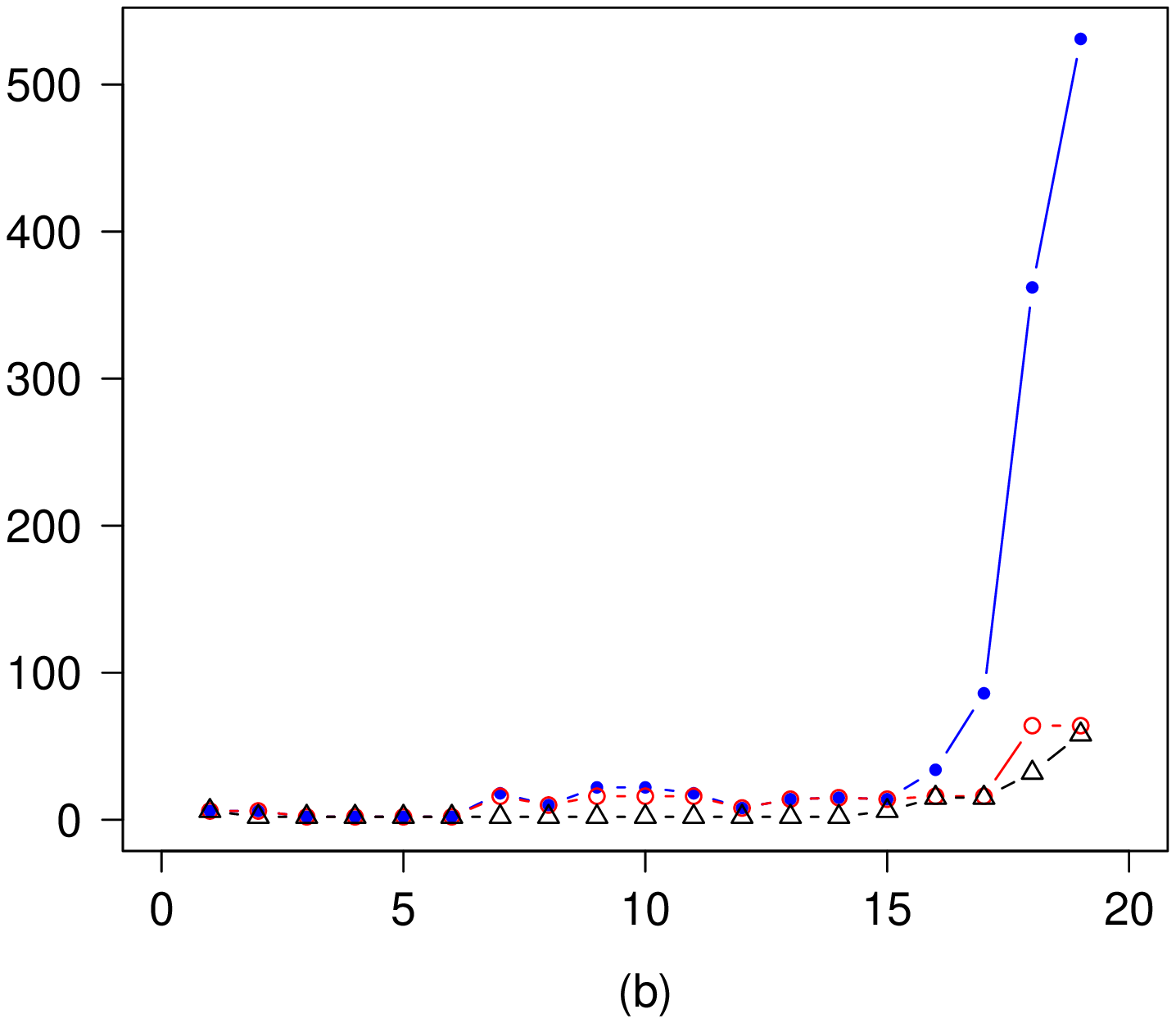}
\par\end{centering}

\caption{\label{fig:Schulz-iteration-stochastic-pde}Results for Example~\ref{sub:Stochastic-PDE-in-8d}.
Schulz error $\left\Vert \mathbb{E}_{n}\right\Vert =\frac{1}{2}\left(\left\Vert \mathbb{I}-\mathbb{X}_{n}\mathbb{B}\right\Vert +\left\Vert \mathbb{I}-\mathbb{B}\mathbb{X}_{n}\right\Vert \right)/\left\Vert I\right\Vert $
per iteration $n=1,2,3,\dots$ on $\log_{10}$ scale and separation
rank of Schulz iterate before reduction (dots) and after (circles)
tensor ID. The triangles show the separation rank after applying ALS
(with a fixed number of iterations) to the result of tensor ID.}

\end{figure}

\subsection{A limitation of tensor ID: orthogonal decompositions}

Finally, we illustrate a limitation of using the tensor ID by constructing
an example for which it is not expected to work at all. Consider the
$d$-dimensional tensor
\begin{equation}
\mathcal{U}=\bigotimes_{j=1}^{d}\left(\sum_{l=1}^{L}\sigma_{l}\mathbf{u}_{j}^{(l)}\right)\label{eq:example-3-orth}
\end{equation}
where $\sigma_{l}=1$ for all $l=1,\dots,L$, and 
\begin{equation}
\langle\mathbf{u}_{j}^{(l)},\mathbf{u}_{j}^{(m)}\rangle=\delta_{lm}\label{eq:example-3-iprod}
\end{equation}
for $j=1,\dots,d$ and $l,m=1,\dots L$. Objects such as (\ref{eq:example-3-orth})
appear quite frequently in various fields. For example, if we take
the vectors $\mathbf{u}_{j}^{(l)}$ to correspond to orthogonal polynomials
with respect to a specified probability measure, discretized at properly
chosen quadrature nodes, $\mathcal{U}$ may be interpreted as a tensor
order polynomial chaos expansion (PCE) \cite{WIENER:1938}. 

Expanding the tensor product (\ref{eq:example-3-orth}), the result
is seen to have nominal separation rank $L^{d}$. By (\ref{eq:example-3-iprod}),
all of the terms are mutually orthogonal and, thus, the columns of
corresponding matrix $U$ in (\ref{eq:umat}) are orthonormal. Consequently,
the tensor ID cannot provide a reduction of the separation rank. However,
if the $s$-values decay rapidly, we can truncate with controlled
error by using Parseval's identity (cf., Example~\ref{sub:Example-of-spectral-properties}).
In fact, truncation may in this case be viewed as a special instance
of the tensor ID, with skeleton indices corresponding to the $k$
terms with $s$-values above the accuracy threshold.

\section*{Acknowledgements}

We would like to thank Martin Mohlenkamp (Ohio University) and Terry
Haut (LANL) for making many useful suggestions to improve the manuscript.

\bibliographystyle{plain}

\include{supplement-rproj}

\end{document}

%% file: supplement-rproj.tex
\section{Online supplement \label{sec:Appendix A}}

\subsection{Proof of Theorem \ref{thm:gram-id}}
\begin{lem}
\label{lem:acs}Suppose that $B_{k}=\bcs P$ is a rank-$k$ interpolative
decomposition of $B$. Then 
\begin{equation}
B_{k}=A^{*}\acs P.\label{eq: bk =00003D a* ac p}
\end{equation}
\end{lem}
\begin{proof}
By definition, we have that 
\begin{equation}
[B]_{ll'}=[A^{*}A]_{ll'}=\langle\mathbf{a}^{(l)},\mathbf{a}^{(l')}\rangle,\,\,\, l,l'=1,\dots,n,\label{eq: b diag elements}
\end{equation}
where $\mathbf{a}^{(l)}$ denotes the column $l$ of $A$. Thus, the
column skeleton of $B$ is of the form
\begin{equation}
[\bcs]_{lj}=\langle\mathbf{a}^{(l)},\mathbf{a}^{(l_{j}')}\rangle,\,\,\,\, j=1,\dots,k,\label{eq: bc elements}
\end{equation}
where indices $l_{j}'$ are those of the skeleton columns of $B$;
we denote this subset as $\mathcal{L}_{k}$. We have
\[
[B_{k}]_{ll'}=\sum_{j=1}^{k}[\bcs]_{lj}[P]_{jl'}=\sum_{j=1}^{k}\langle\mathbf{a}^{(l)},\mathbf{a}^{(l'_{j})}\rangle[P]_{jl'}=\langle\mathbf{a}^{(l)},\sum_{j=1}^{k}\mathbf{a}^{(l'_{j})}[P]_{jl'}\rangle=[A^{*}\acs P]_{ll'}
\]
\end{proof}
\begin{lem}
\label{lem:id-symmetric}Suppose $B$ is a symmetric matrix that admits
a rank-$k$ interpolative decomposition $\hat{B}_{k}=\bcs\hat{P}$
such that
\begin{equation}
\left\Vert B-\hat{B}_{k}\right\Vert _{2}\leq\epsilon_{k}.\label{eq: b minus bk hat norm}
\end{equation}
Then it also admits a decomposition of the form 
\begin{equation}
\hat{B}_{k}=P^{*}\bs\hat{P},\label{eq: bk hat}
\end{equation}
with identical error. Symmetrizing with respect to $P=\left[\begin{array}{c|c}
I & S\end{array}\right]$, i.e., setting
\begin{equation}
B_{k}=P^{*}\bs P,\label{eq: bk}
\end{equation}
this approximation satisfies
\begin{equation}
\left\Vert B-B_{k}\right\Vert _{2}\leq(1+\sqrt{nk(n-k)})\epsilon_{k}.\label{eq: b minus bk norm}
\end{equation}
\end{lem}
\begin{proof}
Define the $n\times n$ matrix $Y=B-B_{k}.$ Following the proof of
Theorem 3 in \cite{C-G-M-R:2005}, we first compute a pivoted QR factorization
of $B$ such that
\begin{equation}
BP_{c}=QR,\label{eq: b pc =00003D q r}
\end{equation}
where the $n\times k$ matrix $Q$ has orthonormal columns, $R$ is
a $k\times n$ upper triangular matrix, and $P_{c}$ is an $n\times n$
permutation matrix. We label blocks of the matrices $Q$ and $R$
as 
\begin{equation}
Q=\left[\begin{array}{c|c}
Q_{11} & Q_{12}\\
\hline Q_{21} & Q_{22}
\end{array}\right],\,\,\,\,\,\,\,\,\, R=\left[\begin{array}{c|c}
R_{11} & R_{12}\\
\hline 0 & R_{22}
\end{array}\right].\label{eq:id-qr}
\end{equation}
Here the blocks of $Q$ have the following dimensions dimensions:
$Q_{11}$ is $k\times k$, $Q_{12}$ is $k\times(n-k)$, $Q_{21}$
is $(n-k)\times k$ and $Q_{22}$ is $(n-k)\times(n-k)$. The blocks
of $R$ have the corresponding dimensions so that the product $QR$
may be expressed in terms of the products of blocks. We set (as in
\cite{C-G-M-R:2005}) 
\begin{equation}
\bcs=\left[\begin{array}{c}
Q_{11}R_{11}\\
\hline Q_{21}R_{11}
\end{array}\right],\label{eq:id-skel-and-error-defn}
\end{equation}
so that $B=\bcs\hat{P}+Y,$ where

\begin{equation}
Y=\left[\begin{array}{c|c}
0 & Q_{12}\\
\hline 0 & Q_{22}
\end{array}\right]R_{22}P_{c}^{*},\label{eq: ymat in appendix proof}
\end{equation}
and 
\begin{equation}
\hat{P}=\left[\begin{array}{c|c}
I & T\end{array}\right]P_{c}^{*}\label{eq:pmat-hat}
\end{equation}
 and $T$ the least squares solution of $R_{11}T=R_{12}$ ($T$ may
not be unique if $R_{11}$ is ill-conditioned). 

Since $B$ is symmetric, a suitable basis for the column space of
$B$ is also a suitable basis for the row space. We compute the QR
decomposition of $\bcs^{*}P_{c}^{*}$ (without pivoting),

\begin{equation}
\bcs^{*}P_{c}^{*}=\tilde{Q}\left[\begin{array}{c|c}
\tilde{R}_{11} & \tilde{R}_{12}\end{array}\right]\label{eq:bcs-qr-1}
\end{equation}
and notice that the blocks $\tilde{R}_{21}$ and $\tilde{R}_{22}$
are zero since $\bcs$ has exactly rank $k$. Since $B_{CS}=A^{*}\acs$,
using Lemma~\ref{lem:acs} we have
\begin{equation}
\bcs^{*}P_{c}^{*}=\acs^{*}\left[\begin{array}{c|c}
\acs & \hat{A}\end{array}\right],\label{eq:bcs-qr-2}
\end{equation}
where the $n\times(n-k)$ block $\hat{A}$ consist of the non-skeleton
columns of $A$. Because the columns have been already pivoted, it
follows that
\begin{eqnarray}
\tilde{Q}\tilde{R}_{11} & = & \acs^{*}\acs\nonumber \\
\tilde{Q}\tilde{R}_{12} & = & \acs^{*}\hat{A}.\label{eq:bcs-qr-3}
\end{eqnarray}
Finally, again following \cite{C-G-M-R:2005}, we have
\begin{eqnarray}
B & = & P_{c}\left[\begin{array}{c}
I\\
\hline \tilde{R}_{12}^{*}(\tilde{R}_{11}^{*})^{-1}
\end{array}\right]\tilde{R}_{11}^{*}\tilde{Q}^{*}P+Y\nonumber \\
 & = & P^{*}\bs\hat{P}+Y\label{eq:id-full-skeleton}
\end{eqnarray}
where $\bs=\acs^{*}\acs$ is the skeleton of $B$, 
\begin{equation}
P=\left[\begin{array}{c|c}
I & S\end{array}\right]P_{c}^{*}.\label{eq:id-p-tilde-matrix}
\end{equation}
and $S$ is a least squares solution of 
\begin{eqnarray}
S\tilde{R}_{11}^{*} & = & \tilde{R}_{12}^{*}.\label{eq: s r11* =00003D r12*}
\end{eqnarray}
satisfying
\begin{equation}
\left\Vert S\right\Vert _{2}\leq\sqrt{nk(n-k)}.\label{eq:id-smatrix-norm-bound}
\end{equation}
 Notice that the error is unchanged upon factorization of the row
space. This proves the first claim.

To prove the second claim, notice first that the left (row) coefficient
matrix $S$ solves
\begin{equation}
\acs^{*}\acs S-\acs^{*}\hat{A}=0,\label{eq:id-left-normal-equations}
\end{equation}
while the right (column) coefficient matrix $T$ solves
\begin{equation}
\left[\begin{array}{c}
\acs^{*}\acs\\
\hline \hat{A}^{*}\acs
\end{array}\right]T=\left[\begin{array}{c}
\acs^{*}\hat{A}\\
\hline \hat{A}^{*}\hat{A}
\end{array}\right]\label{eq:id-right-normal-equations}
\end{equation}
approximately, such that 
\begin{eqnarray}
\left\Vert \acs^{*}\acs T-\acs^{*}\hat{A}\right\Vert _{2} & \leq & \epsilon_{k},\nonumber \\
\bigl\Vert\hat{A}^{*}\acs T-\hat{A}^{*}\hat{A}\bigr\Vert_{2} & \leq & \epsilon_{k}.\label{eq:column-id-submatrix-error-bounds}
\end{eqnarray}
The quantity of interest is $B-\tilde{B}_{k}$, which we permute so
that the skeleton columns are in the first $k$ positions, giving
\begin{eqnarray}
P_{c}^{*}\left(B-P^{*}\bs P\right)P_{c} & = & \left[\begin{array}{c}
\acs^{*}\\
\hline \hat{A}^{*}
\end{array}\right]\left[\begin{array}{c|c}
\acs & \hat{A}\end{array}\right]-\left[\begin{array}{c}
I\\
\hline S^{*}
\end{array}\right]\acs^{*}\acs\left[\begin{array}{c|c}
I & S\end{array}\right]\nonumber \\
 & = & \left[\begin{array}{c|c}
0 & 0\\
\hline 0 & \hat{A}^{*}\hat{A}-S^{*}\acs^{*}\acs S
\end{array}\right]\label{eq: big matrix expression 1}
\end{eqnarray}
where the off-diagonal blocks are set to zero by virtue of the fact
that $S$ solves the normal equations exactly. Adding and subtracting
$S^{*}\acs^{*}\acs T$ to the non-zero block, we have
\begin{eqnarray}
\bigl\Vert\hat{A}^{*}\hat{A}-S^{*}\acs^{*}\acs S\bigr\Vert_{2} & = & \bigl\Vert\hat{A}^{*}\hat{A}-S^{*}\bs S+S^{*}\bs T-S^{*}\bs T\bigr\Vert_{2}\nonumber \\
 & \leq & \bigl\Vert\hat{A}^{*}\hat{A}-S^{*}\bs T\bigr\Vert_{2}+\bigl\Vert S^{*}\bs S-S^{*}\bs T\bigr\Vert_{2}\nonumber \\
 & \leq & \epsilon_{k}+\bigl\Vert S^{*}\bs S-S^{*}\bs T\bigr\Vert_{2}\label{eq: big matrix expression 2}
\end{eqnarray}
where the second line follows from the triangle inequality and the
third from (\ref{eq:column-id-submatrix-error-bounds}). Adding and
subtracting $S^{*}\acs^{*}\hat{A}$, we have
\begin{eqnarray}
\bigl\Vert S^{*}\bs S-S^{*}\bs T\bigr\Vert_{2} & = & \bigl\Vert(S^{*}\bs S-S^{*}\acs^{*}\hat{A})-(S^{*}\bs T-S^{*}\acs^{*}\hat{A})\bigr\Vert_{2}\nonumber \\
 & \leq & \bigl\Vert S\bigr\Vert_{2}\left(\bigl\Vert\bs S-\acs\hat{A}\bigr\Vert_{2}+\bigl\Vert\bs T-\acs^{*}\hat{A}\bigr\Vert_{2}\right)\nonumber \\
 & \leq & \bigl\Vert S\bigr\Vert_{2}\epsilon_{k}\label{eq: big matrix expression 3}
\end{eqnarray}
by the triangle inequality, (\ref{eq:id-left-normal-equations}) and
(\ref{eq:column-id-submatrix-error-bounds}). Hence, applying (\ref{eq:id-smatrix-norm-bound}),
we conclude
\begin{equation}
\bigl\Vert B-B_{k}\bigr\Vert_{2}\leq(1+\bigl\Vert S\bigr\Vert_{2})\epsilon_{k}\leq(1+\sqrt{nk(n-k)})\epsilon_{k}.\label{eq: b minus bk norm in proof}
\end{equation}
\end{proof}
\begin{cor}
\label{cor:P-projector}The coefficient matrix $P$ in (\ref{eq:id-p-tilde-matrix})
solves the normal equations
\begin{eqnarray}
\acs^{*}\acs P & = & \acs^{*}A.\label{eq:id-normal-equations}
\end{eqnarray}
and, thus, minimizes the residual error
\begin{equation}
\left\Vert \acs P-A\right\Vert _{2}.\label{eq:acsp-minimum-l2-norm}
\end{equation}
Setting $A_{k}=\acs P$, it follows that the residual $X=A-A_{k}$
is in the null space of $A_{k}^{*}$.\end{cor}
\begin{proof}
We first note that, in order for a matrix $\tilde{S}$ to minimize
\begin{equation}
\left\Vert \acs\tilde{S}-\hat{A}\right\Vert _{2},\label{eq: ac stilde mins a hat norm}
\end{equation}
it must solve the normal equations (\ref{eq:id-normal-equations}).
From (\ref{eq:bcs-qr-3}), applying the (pivot-free) QR algorithm
to this system yields $\acs^{*}\acs=\tilde{Q}\tilde{R}_{11}$ and
$\acs^{*}\hat{A}=\tilde{Q}\tilde{R}_{12}$, and so
\begin{eqnarray}
\tilde{Q}\tilde{R}_{11}\tilde{S} & = & \tilde{S}\tilde{Q}\tilde{R}_{12}\,\,\,\,\,\,\,\,\,\Rightarrow\,\,\,\,\,\,\,\,\,\,\tilde{S}\,=\,\tilde{R}_{11}^{-1}\tilde{R}_{12}.\label{eq: qr11s =00003D sqr12 implies ...}
\end{eqnarray}
Thus $\tilde{S}=S$ in Lemma~\ref{lem:id-symmetric}. Since $P=\left[\begin{array}{c|c}
I & S\end{array}\right]P_{c}^{*}$, the normal equations (\ref{eq:id-normal-equations}) may be augmented
to include the skeleton columns without increasing the error (the
projection is exact for these columns).

The second claim follows from the fact that $P$ minimizes (\ref{eq:acsp-minimum-l2-norm}).\end{proof}
\begin{rem}
It is emphatically \emph{not} the case that replacing $P$ with the
the column-oriented matrix $\hat{P}$ in Corollary~\ref{cor:P-projector}
leads to the same result. In fact, in light of (\ref{eq:id-right-normal-equations})
and (\ref{eq:column-id-submatrix-error-bounds}), it will generally
not be true that $\hat{P}$ solves the normal equations (\ref{eq:id-normal-equations}).

To complete the proof of the theorem, \end{rem}
\begin{lem}
\label{lem:gram-id-error}Suppose that the column-oriented interpolative
decomposition\textbf{ $\hat{B}_{k}=\bcs\hat{P}$ }satisfies $\bigl\Vert B-\hat{B}_{k}\bigr\Vert_{2}\leq\epsilon_{k},$
with $\hat{P}$ defined in (\ref{eq:pmat-hat}), and $B_{k}=P^{*}\bs P$
is its symmetrized version with $P$ defined in (\ref{eq:id-p-tilde-matrix}).
Setting $A_{k}=\acs P$, we have
\begin{equation}
\left\Vert A-A_{k}\right\Vert _{2}=\left\Vert B-B_{k}\right\Vert _{2}^{1/2}\leq\left((1+\sqrt{nk(n-k)})\epsilon_{k}\right)^{1/2}.\label{eq: a minus ak norm in gram id error lemma}
\end{equation}
\end{lem}
\begin{proof}
That a symmetric decomposition exists follow from Lemma~\ref{lem:id-symmetric}.
Formally, let the QR factorization of $A_{k}$ be given by 
\begin{equation}
A_{k}=Q_{k}R_{k}\label{eq: ak =00003D qk rk}
\end{equation}
where $Q_{k}$ is a $m\times k$ matrix with orthogonal columns and
$R_{k}$ is a $k\times n$ upper rectangular matrix. By Corollary~\ref{cor:P-projector},
the orthogonal projector defined by $Q_{k}Q_{k}^{*}$ coincides with
the matrix $P$ in the sense that
\begin{eqnarray}
A_{k}\,\,=\,\,\acs P & = & Q_{k}Q_{k}^{*}A.\label{eq: ak =00003D qk qk* a}
\end{eqnarray}
Next, define the residual 
\begin{equation}
X=A-A_{k}\label{eq: x =00003D a - ak}
\end{equation}
and decompose it via $X=X_{1}+X_{2}$, where
\begin{equation}
\begin{array}{ccc}
X_{1}=Q_{k}Q_{k}^{*}X, & \,\,\,\,\,\,\,\,\,\,\,\, & X_{2}=(I-Q_{k}Q_{k}^{*})X.\end{array}\label{eq: x1 and x2}
\end{equation}
By Corollary~\ref{cor:P-projector}, we have that $X_{1}=0$, and
so $A=A_{k}+X_{2}$ with $X_{2}$ in the null space of $A_{k}^{*}$.
In addition, 
\begin{eqnarray}
B_{k} & = & A_{k}^{*}A_{k}\,\,=\,\,(Q_{k}Q_{k}^{*}A)^{*}(Q_{k}Q_{k}^{*}A)\,\,=\,\, A^{*}Q_{k}Q_{k}^{*}A.\label{eq: bk defined via orth proj}
\end{eqnarray}
Thus
\begin{eqnarray}
B-B_{k} & = & A^{*}A-A^{*}Q_{k}Q_{k}^{*}A\,\,=\,\,(A-Q_{k}Q_{k}^{*}A)^{*}(A-Q_{k}Q_{k}^{*}A)\nonumber \\
 & = & (A-A_{k})^{*}(A-A_{k}).\label{eq: b - bk =00003D (a - ak)*(a - ak)}
\end{eqnarray}
Hence, $\left\Vert A-A_{k}\right\Vert _{2}=\left\Vert B-B_{k}\right\Vert _{2}^{1/2}$,
with the upper bound $\left((1+\sqrt{nk(n-k)})\epsilon_{k}\right)^{1/2}$
following directly from Lemma~\ref{lem:id-symmetric}.\end{proof}

%% file: rproj.bbl
\begin{thebibliography}{10}

\bibitem{ACA-BUL:2009}
E.~Acar and B.~Yener.
\newblock Unsupervised multiway data analysis: A literature survey.
\newblock {\em Knowledge and Data Engineering, IEEE Transactions on},
  21(1):6--20, 2009.

\bibitem{ACHLIO:2003}
D.~Achlioptas.
\newblock Database-friendly random projections: {J}ohnson-{L}indenstrauss with
  binary coins.
\newblock {\em J. Comput. Syst. Sci.}, 66(4):671--687, June 2003.

\bibitem{ACH-MCS:2007}
D.~Achlioptas and F.~Mcsherry.
\newblock Fast computation of low-rank matrix approximations.
\newblock {\em Journal of the ACM (JACM)}, 54(2):9, 2007.

\bibitem{AIL-CHA:2006}
N.~Ailon and B.~Chazelle.
\newblock Approximate nearest neighbors and the fast {J}ohnson-{L}indenstrauss
  transform.
\newblock In {\em Proceedings of the thirty-eighth annual ACM symposium on
  Theory of computing}, pages 557--563. ACM, 2006.

\bibitem{TENSOR:2012}
B.~W. Bader and T.~G. Kolda.
\newblock Matlab {T}ensor {T}oolbox {V}ersion 2.5.
\newblock Available online, January 2012.

\bibitem{BEYLKI:1992}
G.~Beylkin.
\newblock On the representation of operators in bases of compactly supported
  wavelets.
\newblock {\em SIAM J. Numer. Anal.}, 29(6):1716--1740, 1992.

\bibitem{BEYLKI:1994}
G.~Beylkin.
\newblock On wavelet-based algorithms for solving differential equations.
\newblock In J.~J. Benedetto and M.~W. Frazier, editors, {\em Wavelets:
  mathematics and applications}, Stud. Adv. Math., pages 449--466. CRC, Boca
  Raton, FL, 1994.

\bibitem{BE-CO-RO:1991}
G.~Beylkin, R.~Coifman, and V.~Rokhlin.
\newblock Fast wavelet transforms and numerical algorithms, {I}.
\newblock {\em Comm. Pure Appl. Math.}, 44(2):141--183, 1991.
\newblock {Y}ale Univ. Technical Report YALEU/DCS/RR-696, August 1989.

\bibitem{BE-CO-RO:1992}
G.~Beylkin, R.~Coifman, and V.~Rokhlin.
\newblock Wavelets in numerical analysis.
\newblock In {\em Wavelets and their applications}, pages 181--210. Jones and
  Bartlett, Boston, MA, 1992.

\bibitem{B-F-H-K-M:2012}
G.~Beylkin, G.~Fann, R.~J. Harrison, C.~Kurcz, and L.~Monz\'on.
\newblock Multiresolution representation of operators with boundary conditions
  on simple domains.
\newblock {\em Appl. Comput. Harmon. Anal.}, 33:109--139, 2012.
\newblock http://dx.doi.org/10.1016/j.acha.2011.10.001.

\bibitem{BE-GA-MO:2009}
G.~Beylkin, J.~Garcke, and M.~J. Mohlenkamp.
\newblock Multivariate regression and machine learning with sums of separable
  functions.
\newblock {\em SIAM Journal on Scientific Computing}, 31(3):1840--1857, 2009.

\bibitem{BEY-MOH:2002}
G.~Beylkin and M.~J. Mohlenkamp.
\newblock Numerical operator calculus in higher dimensions.
\newblock {\em Proc. Natl. Acad. Sci. USA}, 99(16):10246--10251, August 2002.

\bibitem{BEY-MOH:2005}
G.~Beylkin and M.~J. Mohlenkamp.
\newblock Algorithms for numerical analysis in high dimensions.
\newblock {\em SIAM J. Sci. Comput.}, 26(6):2133--2159, July 2005.

\bibitem{BIAGIO:2012}
D.~Biagioni.
\newblock {\em Numerical {C}onstruction of {G}reen's {F}unctions in
  {H}igh-dimensional {E}lliptic {P}roblems with {V}ariable {C}oefficients and
  {A}nalysis of {R}enewable {E}nergy {D}ata via {S}parse and {S}eparable
  {A}pproximations}.
\newblock PhD thesis, University of Colorado, 2012.

\bibitem{BRO:1997}
R.~Bro.
\newblock Parafac. {T}utorial \& {A}pplications.
\newblock In {\em Chemom. {I}ntell. {L}ab. {S}yst., {S}pecial {I}ssue 2nd
  {I}nternet {C}onf. in {C}hemometrics (incinc'96)}, volume~38, pages 149--171,
  1997.
\newblock
  \texttt{http://www.models.kvl.dk/users/rasmus/presentations/parafac\_tutoria%
l/paraf.htm}.

\bibitem{CAN-TAO:2006}
E.~J. Cand{\`e}s and T.~Tao.
\newblock Near-optimal signal recovery from random projections: universal
  encoding strategies?
\newblock {\em IEEE Trans. Inform. Theory}, 52(12):5406--5425, 2006.

\bibitem{CA-RO-TA:2006a}
E.J. Cand\`es, J.~Romberg, and T.~Tao.
\newblock Robust uncertainty principles: exact signal reconstruction from
  highly incomplete frequency information.
\newblock {\em Information Theory, IEEE Transactions on}, 52(2):489 -- 509,
  feb. 2006.

\bibitem{CAR-CHA:1970}
J.~D. Carroll and J.~J. Chang.
\newblock Analysis of individual differences in multidimensional scaling via an
  {N}-way generalization of {E}ckart-{Y}oung decomposition.
\newblock {\em Psychometrika}, 35:283--320, 1970.

\bibitem{C-G-M-R:2005}
H.~Cheng, Z.~Gimbutas, P.-G. Martinsson, and V.~Rokhlin.
\newblock On the compression of low-rank matrices.
\newblock {\em SIAM Journal of Scientific Computing}, 205(1):1389--1404, 2005.

\bibitem{V-K-K-V:2005}
W.~F. de~la Vega, M.~Karpinski, R.~Kannan, and S.~Vempala.
\newblock Tensor decomposition and approximation schemes for constraint
  satisfaction problems.
\newblock In {\em Proceedings of the thirty-seventh annual ACM symposium on
  Theory of computing}, STOC '05, pages 747--754, New York, NY, USA, 2005. ACM.

\bibitem{LA-MO-VA:2000}
L.~De~Lathauwer, B.~De~Moor, and J.~Vandewalle.
\newblock On the best rank-1 and rank-(r1,r2,. . .,rn) approximation of
  higher-order tensors.
\newblock {\em SIAM J. Matrix Anal. Appl.}, 21(4):1324--1342, 2000.

\bibitem{SIL-LIM:2008}
V.~de~Silva and L.-H. Lim.
\newblock Tensor rank and the ill-posedness of the best low-rank approximation
  problem.
\newblock {\em SIAM J. Matrix Anal. Appl.}, 30(3):1084--1127, 2008.

\bibitem{DEMMEL:1999}
J.~Demmel.
\newblock Accurate singular value decompositions of structured matrices.
\newblock {\em SIAM J. Matrix Anal. Appl.}, 21(2):562--580, 1999.

\bibitem{DRI-KAN:2003}
P.~Drineas and R.~Kannan.
\newblock Pass efficient algorithms for approximating large matrices.
\newblock In {\em Proceedings of the fourteenth annual ACM-SIAM symposium on
  Discrete algorithms}, SODA '03, pages 223--232, Philadelphia, PA, USA, 2003.
  Society for Industrial and Applied Mathematics.

\bibitem{DRI-MAH:2007}
P.~Drineas and M.~W. Mahoney.
\newblock A randomized algorithm for a tensor-based generalization of the
  {S}ingular {V}alue {D}ecomposition.
\newblock {\em Linear Algebra Appl.}, 420(2-3):553--571, 2007.

\bibitem{FR-KA-VE:2004}
A.~Frieze, R.~Kannan, and S.~Vempala.
\newblock Fast {M}onte-{C}arlo algorithms for finding low-rank approximations.
\newblock {\em J. ACM}, 51(6):1025--1041, November 2004.

\bibitem{GO-TY-ZA:1997}
Sergei~A Goreinov, Eugene~E Tyrtyshnikov, and Nickolai~L Zamarashkin.
\newblock A theory of pseudoskeleton approximations.
\newblock {\em Linear Algebra and its Applications}, 261(1):1--21, 1997.

\bibitem{GO-ZA-TY:1997}
Sergei~A Goreinov, Nikolai~Leonidovich Zamarashkin, and Evgenii~Evgen'evich
  Tyrtyshnikov.
\newblock Pseudo-skeleton approximations by matrices of maximal volume.
\newblock {\em Mathematical Notes}, 62(4):515--519, 1997.

\bibitem{GR-KR-TO:2013}
L.~Grasedyck, D.~Kressner, and C.~Tobler.
\newblock A literature survey of low-rank tensor approximation techniques.
\newblock {\em CoRR}, abs/1302.7121, 2013.

\bibitem{GU-EIS:1996}
Ming Gu and Stanley~C Eisenstat.
\newblock Efficient algorithms for computing a strong rank-revealing qr
  factorization.
\newblock {\em SIAM Journal on Scientific Computing}, 17(4):848--869, 1996.

\bibitem{HA-MA-TR:2011}
N.~Halko, P.-G. Martinsson, and J.~A. Tropp.
\newblock Finding structure with randomness: probabilistic algorithms for
  constructing approximate matrix decompositions.
\newblock {\em SIAM Review}, 53(2):217--288, 2011.

\bibitem{HARSHM:1970}
R.~A. Harshman.
\newblock Foundations of the {P}arafac procedure: model and conditions for an
  ``explanatory'' multi-mode factor analysis.
\newblock Working Papers in Phonetics~16, UCLA, 1970.
\newblock \texttt{http://publish.uwo.ca/$\sim$harshman/wpppfac0.pdf}.

\bibitem{HIL-LIM:2009}
C.~Hillar and L.-H. Lim.
\newblock Most tensor problems are {NP} hard.
\newblock Technical Report arXiv:0911.1393, Nov 2009.

\bibitem{JOH-LIN:1984}
W.~B. Johnson and J.~Lindenstrauss.
\newblock Extensions of {L}ipschitz mappings into a {H}ilbert space.
\newblock In {\em Conference in modern analysis and probability ({N}ew {H}aven,
  {C}onn., 1982)}, volume~26 of {\em Contemp. Math.}, pages 189--206. Amer.
  Math. Soc., Providence, RI, 1984.

\bibitem{KOF-REG:2001}
E.~Kofidis and P.~A. Regalia.
\newblock On the best rank-1 approximation of higher-order supersymmetric
  tensors.
\newblock {\em SIAM J. Matrix Anal. Appl.}, 23(3):863--884, 2001.

\bibitem{KOL-BAD:2009}
T.~G. Kolda and B.~W. Bader.
\newblock Tensor decompositions and applications.
\newblock {\em SIAM Review}, 51(3):455--500, 2009.

\bibitem{KOL-MAY:2011}
T.~G. Kolda and J.~R. Mayo.
\newblock Shifted power method for computing tensor eigenpairs.
\newblock {\em SIAM Journal on Matrix Analysis and Applications},
  32(4):1095--1124, October 2011.

\bibitem{L-W-M-R-T:2007}
E.~Liberty, F.~Woolfe, P-G. Martinsson, V.~Rokhlin, and M.~Tygert.
\newblock Randomized algorithms for the low-rank approximation of matrices.
\newblock {\em Proc. Natl. Acad. Sci. USA}, 104(51):20167--20172, 2007.

\bibitem{LIM:2005}
L.-H. Lim.
\newblock Singular values and eigenvalues of tensors: a variational approach.
\newblock In {\em Computational Advances in Multi-Sensor Adaptive Processing,
  2005 1st IEEE International Workshop on}, pages 129--132, 2005.

\bibitem{MAHONE:2006}
M.~W. Mahoney.
\newblock Tensor-cur decompositions for tensor-based data.
\newblock In {\em In Proceedings of the 12th Annual ACM SIGKDD Conference},
  pages 327--336, 2006.

\bibitem{MA-RO-TY:2006}
P-G. Martinsson, V.~Rokhlin, and M.~Tygert.
\newblock A randomized algorithm for the approximation of matrices.
\newblock Technical report, Yale CS research report YALEU/DCS/RR-1361, 2006.

\bibitem{MA-RO-TY:2011}
P.-G. Martinsson, V.~Rokhlin, and M.~Tygert.
\newblock A randomized algorithm for the approximation of matrices.
\newblock {\em Applied and Computational Harmonic Analysis}, 30(1):47--68,
  2011.

\bibitem{MATLAB:2012}
MATLAB.
\newblock {\em Version 8.0.0 (R2012b)}.
\newblock The MathWorks Inc., Natick, Massachusetts, 2012.

\bibitem{MOHLEN:2011}
Martin~J Mohlenkamp.
\newblock Musings on multilinear fitting.
\newblock {\em Linear Algebra and its Applications}, 438:834--852, 2011.

\bibitem{NG-DR-TR:2010}
N.~H Nguyen, P.~Drineas, and T.~D Tran.
\newblock Tensor sparsification via a bound on the spectral norm of random
  tensors.
\newblock {\em arXiv preprint arXiv:1005.4732}, 2010.

\bibitem{P-T-R-V:1998}
C.~H. Papadimitriou, H.~Tamaki, P.~Raghavan, and S.~Vempala.
\newblock Latent semantic indexing: a probabilistic analysis.
\newblock In {\em Proceedings of the seventeenth ACM SIGACT-SIGMOD-SIGART
  symposium on Principles of database systems}, PODS '98, pages 159--168, New
  York, NY, USA, 1998. ACM.

\bibitem{RA-SC-ST:2013}
H.~Rauhut, R.~Schneider, and Z.~Stojanac.
\newblock Low-rank tensor recovery via iterative hard thresholding.
\newblock In {\em Proceedings of the International Conference on Sampling
  Theory and Applications (preprint)}, 2013.

\bibitem{RUD-VER:2007}
M.~Rudelson and R.~Vershynin.
\newblock Sampling from large matrices: An approach through geometric
  functional analysis.
\newblock {\em J. ACM}, 54(4), July 2007.

\bibitem{SARLOS:2006}
T.~Sarlos.
\newblock Improved approximation algorithms for large matrices via random
  projections.
\newblock In {\em Foundations of Computer Science, 2006. FOCS '06. 47th Annual
  IEEE Symposium on}, pages 143--152, 2006.

\bibitem{SCHULZ:1933}
G.~Schulz.
\newblock Iterative {B}erechnung der reziproken {M}atrix.
\newblock {\em Z. Angew. Math. Mech.}, 13:57--59, 1933.

\bibitem{TOM-BRO:2006}
G.~Tomasi and R.~Bro.
\newblock A comparison of algorithms for fitting the {PARAFAC} model.
\newblock {\em Comput. Statist. Data Anal.}, 50(7):1700--1734, 2006.

\bibitem{TSOURA:2009}
C.~E. Tsourakakis.
\newblock Mach: Fast randomized tensor decompositions.
\newblock {\em CoRR}, abs/0909.4969, 2009.

\bibitem{TYRTYS:1996}
Eugene Tyrtyshnikov.
\newblock Mosaic-skeleton approximations.
\newblock {\em Calcolo}, 33(1-2):47--57, 1996.

\bibitem{WIENER:1938}
N.~Wiener.
\newblock The {H}omogeneous {C}haos.
\newblock {\em American Journal of Mathematics}, 60(4):897--936, 1938.

\bibitem{W-L-R-T:2008}
F.~Woolfe, E.~Liberty, V.~Rokhlin, and M.~Tygert.
\newblock A fast randomized algorithm for the approximation of matrices.
\newblock {\em Appl. Comput. Harmon. Anal.}, 25(3):335--366, 2008.

\bibitem{ZHA-GOL:2001}
T.~Zhang and G.~H. Golub.
\newblock Rank-one approximation to high order tensors.
\newblock {\em SIAM J. Matrix Anal. Appl.}, 23(2):534--550, 2001.

\end{thebibliography}
